\documentclass[12pt]{amsart}
\usepackage{amssymb,amsmath,amsthm,bm}
\usepackage[colorinlistoftodos,prependcaption,textsize=tiny]{todonotes}
 \definecolor{darkblue}{RGB}{0,0,160}
\usepackage{fouriernc} 
\usepackage[colorlinks=true,allcolors=darkblue]{hyperref}
\DeclareSymbolFont{usualmathcal}{OMS}{cmsy}{m}{n}
\DeclareSymbolFontAlphabet{\mathcal}{usualmathcal}

\usepackage[T1]{fontenc}

\usepackage{color}

\usepackage{parskip}
\usepackage{setspace}

\usepackage{amsmath,amsthm,amscd,amssymb, mathrsfs}

\usepackage{latexsym}

\newtheorem{theorem}{Theorem}[section]
\newtheorem{corollary}[theorem]{Corollary}
\newtheorem{lemma}[theorem]{Lemma}
\newtheorem{proposition}[theorem]{Proposition}
\newtheorem{definition}[theorem]{Definition}
\newtheorem{question}[theorem]{Question}
\newtheorem{remark}[theorem]{Remark}

\newtheorem{conjecture}[theorem]{Conjecture}
\newtheorem*{definition*}{Definition}

\title[\parbox{14cm}{\centering{Extension theorems and connections\hspace{1in}}} \quad]{Extension theorems and a connection to the Erd\H{o}s-Falconer
distance problem over finite fields}
\author{Doowon Koh,  Thang Pham, and Le Anh Vinh}

\address{Department of Mathematics\\
Chungbuk National University \\
Cheongju, Chungbuk 28644 Korea}
\email{koh131@chungbuk.ac.kr}

\address{Department of Mathematics\\
ETH Switzerland}
\email{phamanhthang.vnu@gmail.com}

\address{Department of Mathematics\\
Vietnam National University}
\email{vinhla@vnu.edu.vn}
\subjclass[2010]{42B05, 11T23, 52C10}
\begin{document}
\begin{abstract}
The first purpose of this paper is to provide new finite field extension theorems for paraboloids and spheres. By using the unusual good Fourier transform of the zero sphere in some specific dimensions, which has been discovered recently in the work of  Iosevich, Lee, Shen, and the first and second listed authors (2018), we provide a new  $L^2\to L^r$ extension estimate for paraboloids in dimensions $d=4k+3$ and $q\equiv 3\mod 4$, which  improves significantly  the recent exponent obtained by the first listed author. In the case of spheres, we introduce a way of using \textit{the first association scheme graph} to analyze energy sets, and as a consequence, we obtain new $L^p\to L^4$ extension theorems for spheres of primitive radii in odd dimensions, which break the Stein-Tomas result toward $L^p\to L^4$ which has stood for more than ten years. Most significantly, it follows from the results for spheres that there exists a different extension phenomenon between spheres
and paraboloids in odd dimensions, namely, the $L^p\to L^4$ estimates for spheres with primitive radii are much stronger than those for paraboloids. Based on new estimates,  we will also clarify conjectures on finite field extension problem for spheres. This results in a reasonably complete description of finite field extension theorems for spheres. The second purpose is to show that there is a connection between the restriction conjecture associated to paraboloids  and the Erd\H{o}s-Falconer distance conjecture over finite fields. The last is to prove that the Erd\H{o}s-Falconer distance conjecture holds in odd dimensional spaces when we study distances between two sets: one set lies on a variety (paraboloids or spheres), and the other set is arbitrary in $\mathbb{F}_q^d$. 
\end{abstract}
\maketitle
\section{Introduction}
\subsection{Extension theorems for varieties}
Let $d\sigma$ be a surface measure supported on a compact subset $V$ in $\mathbb R^d.$ 
The Fourier restriction problem for $V$ is to determine the exponents $1\le p, r \le \infty$ such that 
the restriction inequality
$$ \|\widehat{f}\|_{L^r(V, d\sigma)} \le C \|f\|_{L^p(\mathbb R^d)}$$
holds for all Schwartz functions $f.$ 
This problem  is one of the most important open problems in harmonic analysis and it has several applications to other fields.
A well-known dual argument shows that the restriction inequality is the same as the following extension estimate
$$ \|(gd\sigma)^\vee\|_{L^{p'}(\mathbb{R}^d)} \le C \|g\|_{L^{r'}(V, d\sigma)},$$
where $p', r'$ denote the H\"{o}lder conjugates of $p, r,$ respectively. Thus the Fourier restriction problem is also called the extension problem.

In this paper, we use the following notations: $X \ll Y$ means that there exists  some absolute constant $C_1>0$ such that $X \leq C_1Y$, $X \gtrsim Y$ means $X\gg (\log Y)^{-C_2} Y$ for some absolute constant $C_2>0$,   $X\sim Y$ means $Y\ll X\ll Y$, and $X\not\in (a, b)$ means either $X\le a$ or $X\ge b$.

Let $q$ be an odd prime power, and $\mathbb{F}_q$ be the finite field of order $q$. In 2002, Mockenhaupt and Tao \cite{MT04} introduced the Fourier restriction problem for algebraic varieties in the finite field setting. Over the last $16$ years, this topic has received a fair amount of study, see \cite{06, IK09, I-K, KS12, european, Le13, RS}. 

Before reviewing this problem, we introduce some notations and basic knowledge  in the discrete Fourier analysis.

Let $\mathbb F_q^d$ be the $d$-dimensional vector space over the finite field $\mathbb{F}_q$. We denote by $\chi$ a non-trivial additive character of $\mathbb F_q.$ Our results in this paper will be independent of the choice of the character $\chi$.
We recall that the orthogonality relation of $\chi$ states that                                                                
$$ \sum_{x\in \mathbb F_q^d} \chi(m\cdot x)=  \left\{\begin{array}{ll}0 \quad &\mbox{if}\quad m\ne (0,\ldots,0)\\
                                                                      q^d  \quad &\mbox{if}\quad m= (0,\ldots,0).\end{array} \right. $$ 
                                                                                           
Given a complex valued function on $\mathbb F_q^d,$  the (\textbf{normalized}) Fourier transform of $f$, denoted by $\widehat{f}$, is defined by
$$ \widehat{f}(m)=\frac{1}{q^d} \sum_{x\in \mathbb F_q^d} \chi(-x\cdot m) f(x).$$
The following Fourier inversion theorem can be easily proved by the orthogonality relation of $\chi$:
$$ f(x)=\sum_{m\in \mathbb F_q^d} \widehat{f}(m) \chi(m\cdot x).$$
By the orthogonality relation of $\chi,$ it follows that
$$ \sum_{m\in \mathbb F_q^d} |\widehat{f}(m)|^2 = q^{-d} \sum_{x\in \mathbb F_q^d} |f(x)|^2,$$
which is referred to as the Plancherel theorem. For example, if $E\subset \mathbb F_q^d,$ then we have
$$\sum_{m\in \mathbb F_q^d} |\widehat{E}(m)|^2=q^{-d} |E|.$$ 
Here and throughout the paper, we identify a set $E$ with the indicator function $1_E$ on $E.$
Now we adopt an unusual notation $\widetilde{f}$ to indicate the Fourier transform of $f$ which does not take the normalizing factor $q^{-d}.$ More precisely, the Fourier transform of $f$, denoted by $\widetilde{f}$, is defined by
$$ \widetilde{f}(x)=\sum_{m\in\mathbb F_q^d} \chi(-m\cdot x) f(m).$$
In addition, the inverse Fourier transform of $f$, denoted by $f^\vee$, is defined by
$$ f^\vee(m)=\frac{1}{q^d} \sum_{x\in \mathbb F_q^d}  \chi(x\cdot m) f(x).$$

Now we introduce the finite field restriction problem. 
Given the vector space $\mathbb F_q^d$ over $\mathbb F_q$, we consider two different measures on $\mathbb F_q^d:$ normalized counting measure  denoted by $dn$ and counting measure denoted by $dc.$ 
Let $V \subset \mathbb F_q^d$ be an algebraic variety which is a set of solutions to a polynomial equation.
We endow the variety $V$ with a normalized surface measure denoted by $d\sigma.$ 
With measures defined above, if $f$ is a complex valued function on $\mathbb F_q^d,$  its integrals are defined as follows:
$$ \int_{\mathbb F_q^d} f(x) ~dn(x) := q^{-d} \sum_{x\in \mathbb F_q^d} f(x),$$
$$ \int_{\mathbb F_q^d} f(m)~ dc(m) := \sum_{m\in \mathbb F_q^d} f(m),$$
$$ \int_{V} f(x) ~d\sigma(x) = \frac{1}{|V|} \sum_{x\in V} f(x),$$
where  $|V|$ denotes the cardinality of the set $V.$
The inverse Fourier transform of the measure $(fd\sigma)$, denoted by $(fd\sigma)^\vee$, is given by
$$ (fd\sigma)^\vee(m):=\int_{V} \chi(m\cdot x) f(x) d\sigma(x)= \frac{1}{|V|} \sum_{x\in V} \chi(m\cdot x) f(x).$$
As usual, notation of norms of functions can be employed. 
For $1\le p,r\le \infty,$ we denote by $R_V^*(p\to r)$ the smallest real number such that the extension inequality
\begin{equation}\label{DefExt} \|(fd\sigma)^\vee\|_{L^r(\mathbb F_q^d, dc)} \le R_V^*(p\to r) \|f\|_{L^p(V, d\sigma)}\end{equation}
holds for all functions $f$ on $V.$ The number $R_V^*(p\to r) $ should be independent of the function $f$ on $V$, but it may depend on the size of the underlying field $\mathbb F_q.$
The extension problem associated to $V$ asks us to determine all exponents $1\le p, r\le \infty$ such that 
$$R_V^*(p\to r)\ll 1,$$ 
where the implicit constant in $\ll$ is independent of $q$, the size of the underlying field $\mathbb F_q.$

By duality, the extension problem is equivalent to the restriction problem which is to determine $1\le p', r'\le \infty$ such that 
the following restriction inequality holds:
$$ \|\widetilde{g}\|_{L^{p'}(V, d\sigma)}\le R_V^*(p\to r) \|g\|_{L^{r'}(\mathbb F_q^d, dc)} \quad \mbox{for all functions}~~g:\mathbb F_q^d \to \mathbb C.$$

From H\"{o}lder's inequality, we have the following trivial bound:
$$ R^*_V(p\to \infty)=1 \quad \mbox{for} \quad 1\le p\le \infty.$$
Since $dc$ is the counting measure on $\mathbb F_q^d$ and $d\sigma$ is the normalized measure on $V,$ we see
$$ \|(fd\sigma)^\vee\|_{L^{r_1}(\mathbb F_q^d, dc)} \le \|(fd\sigma)^\vee\|_{L^{r_2}(\mathbb F_q^d, dc)} \quad \mbox{when} \quad 1\le r_2\le r_1\le \infty,$$
and 
$$\|f\|_{L^{p_1}(V, d\sigma)}\le \|f\|_{L^{p_2}(V, d\sigma)} \quad \mbox{when} \quad 1\le p_1\le p_2\le \infty.$$
From these facts, it follows that for each $1\le p, r\le \infty,$
$$ R^*_V(p\to r_1)\le R^*_V(p\to r_2) \quad \mbox{when} \quad 1\le r_2\le r_1\le \infty,$$
and 
$$ R^*_V(p_1\to r)\le R^*_V(p_2\to r) \quad \mbox{when} \quad 1\le p_2\le p_1\le \infty.$$
Hence, our problem can be reduced to certain endpoint estimates. For example, in order to establish the $L^2\to L^r$ extension estimates for $V,$ it suffices to find the smallest exponent $r$ such that
$R^*_V(2\to r)\ll 1.$  Similarly, the sharp $L^p\to L^4$ estimate can be proved if we find the smallest $p$ such that $R^*_V(p\to 4)\ll 1.$

Necessary  conditions for the bound $R_V^*(p\to r)\ll 1$ can be determined by the sizes of $V$ and an affine subspace lying on $V.$ In fact, Mockenhaupt and Tao \cite{MT04} showed that if the variety $V\subset \mathbb F_q^d$ with $|V|\sim q^{d-1}$ contains an affine subspace $H$ with $|H|=q^k$,  then the necessary conditions for the bound $R_V^*(p\to r)\ll 1$
are given by 
\begin{equation}\label{NecExt} r\ge \frac{2d}{d-1} \quad\mbox{and}\quad r\ge \frac{p(d-k)}{(p-1)(d-1-k)},
\end{equation}
which means that  $(1/p, 1/r)$ lies on the convex hull determined by the following four points:
\begin{equation}\label{Necovex}(0,0),~ \left(0, \frac{d-1}{2d}\right), ~ \left(\frac{d^2-dk-d-k}{2d(d-1-k)}, \frac{d-1}{2d}\right), ~(1,0).\end{equation}

Let $P$ be a paraboloid in $\mathbb{F}_q^d$ defined as follows: 
\[P:=\{(x_1, x_2, \ldots, x_{d-1}, x_d)\colon x_d=x_1^2+\cdots+x_{d-1}^2, (x_1, \ldots, x_{d-1})\in \mathbb{F}_q^{d-1}\}.\]
For $j\in \mathbb{F}_q$, the sphere $S_j$ of radius $j$ centered at the origin in $\mathbb{F}_q^d$ is defined by 
\[S_j:=\{x=(x_1, \ldots, x_d)\in \mathbb{F}_q^d\colon ||x||=x_1^2+\cdots x_d^2=j\}.\]

In this paper, we will focus on the case when $V$ is either the paraboloid $P$ or a sphere of non-zero radius in $\mathbb{F}_q^d$. If the variety $V$ is the paraboloid $P$ or the sphere $S_j$ with $j\ne 0,$ it has been conjectured that the above necessary conditions are also sufficient for the bound $R^*_V(p\to r)\ll 1$, where we takes $k$ as the dimension of a maximal affine subspace lying on $V$.

In two dimensions, since the circle $S_j (j\ne 0)$ and the parabola $P$ do not contain any line, we can take $k=0$ and $d=2.$ Hence, the extension problem for the circle or the parabola on the plane is reduced to proving $R^*_V(2\to 4)\ll 1.$  In two dimensions, the extension conjecture for the parabola and the circle was completely solved by Mockenhaupt and Tao \cite{MT04},  Iosevich and the first listed author \cite{06}, respectively. Their results were extended to arbitrary curves which do not contain any line by the first listed author and Shen \cite{KS12}. However, in higher dimensions, the extension conjecture is still open.

\subsubsection{\textbf{Extension theorems for paraboloids}} 
Over the last ten years, there has been a lot of work aimed at proving estimates beyond the Stein-Tomas result. Unlike the Euclidean case, it turns out that the ``$r$'' index of the standard Stein-Tomas extension theorem $R^*_P(2\to (2d+2)/(d-1)) \ll 1$ can be improved for even dimensions and for certain odd dimensions.

From the necessary conditions \eqref{Necovex}, one can conjecture that to obtain the sharp $L^2\to L^r$ extension estimate for $P$, we only need to prove the critical estimate that $R^*_P(2\to r_2)\ll 1.$ Here,  the critical exponent $r_2$ is defined by
\begin{equation}\label{r-2} r_2:=\frac{2(d^2-dk_*-d+k_*)}{(d-1)(d-1-k_*)},\end{equation}
where $k_*$ denotes the dimension of a maximal subspace lying on the paraboloid $P.$ 

Based on the dimension $k_*$ of a maximal subspace in the paraboloid $P$, the following lemma on exponents $r_2$ was given in \cite{kk}.

\begin{lemma}\label{Conj2}
Let $r_2$ denote the critical exponent defined in \eqref{r-2}. We have
\begin{enumerate}
\item If $d\ge 2$ is even, then  $r_2=\frac{2d+4}{d}.$
\item If $d=4k-1$, $k\in \mathbb N$, and $ -1\in \mathbb F_q$ is not a square number, then
$ r_2= \frac{2d+6}{d+1}.$
\item  If $d=4k+1$, $k \in \mathbb N$, then $ r_2=\frac{2d+2}{d-1}.$
\item  If $d=4k-1$, $k \in \mathbb N$, and $-1\in \mathbb F_q$ is a square number, then
$r_2=\frac{2d+2}{d-1}.$
\end{enumerate}
\end{lemma} 
In \cite{MT04}, Mockenhaupt and Tao proved the Stein-Tomas result,  which says that $R^*_V(2\to (2d+2)/(d-1))\ll 1$ for all dimensions $d\ge 2.$ Hence, in odd dimensions $d$ with the assumptions of the third  or fourth part of  Lemma \ref{Conj2}, the Stein-Tomas result gives the sharp $L^2\to L^r$ extension estimate for the paraboloid $P.$ 
Furthermore,  using the interpolation theorem and the trivial $L^1\to L^\infty$ estimate, they obtained that $R_P^*( \frac{4d-4}{3d-5}\to 4)\ll 1,$ which becomes the sharp $L^p\to L^4$ estimate. The precise statement is as follows.

\begin{proposition} \label{lemP} Let $P\subset \mathbb F_q^d.$ If $d=4k+1$, $k\in \mathbb{N}$ or $d=4k-1$ with $q\equiv 1 \mod 4,$ then
 $$R_P^*(p\to 4)\ll 1 \quad \mbox{if and only if }\quad \frac{4d-4}{3d-5} \le p\le \infty.$$
  
\end{proposition}

Next, notice from the first part of Lemma \ref{Conj2} that for even dimensions $d$, the bound $R^*_P(2\to (2d+4)/d)$ gives the sharp $L^2\to L^r$ extension estimate for $P.$ In two and four dimensions, the sharp bound was proved by Mockenhaupt and Tao \cite{MT04} and Rudnev and Shkredov \cite{RS}, respectively. In addition, the sharp bound for even dimensions $d\ge 8$ was obtained by Iosevich, Lewko, and the first listed author \cite{hello}. In $d=6,$ they also showed that $R^*_P(2\to (2d+4)/d) \lesssim 1,$ which gives the sharp $L^2\to L^r$ bound up to the endpoint. 

Compared to aforementioned cases, it is much harder to prove the sharp $L^2\to L^r$ estimate for $P$ under the assumptions of the second part of Lemma \ref{Conj2}. The following conjecture is based on the second part of Lemma \ref{Conj2}. 
 

\begin{conjecture}\label{L2ConjP} Suppose that $d=4k-1$ for some $k\in \mathbb N$ and $ -1\in \mathbb F_q$ is not a square number. We have
$$ R^*_P\left(2\to \frac{2d+6}{d+1}\right)\ll 1,$$
which gives the sharp $L^2\to L^r$ estimate for the paraboloid $P.$
\end{conjecture} 
Assuming that $q$ is a prime with $q\equiv 3\mod 4$ and $d=3$, Mockenhaupt and Tao \cite{MT04} gave the following $L^2\to L^r$ extension theorem for paraboloid
\[R_P^*\left(2\to \frac{18}{5}+\epsilon\right)\ll 1,\]
for any $\epsilon>0.$ This result has been improved slightly over recent years. For instance, Lewko \cite{Le13} showed that there exists $\epsilon>0$ such that 
\[R_P^*\left(2\to \frac{18}{5}-\epsilon\right)\ll 1.\]
The best current bound is due to Rudnev and Shkredov \cite{RS}, namely,  
\[R_P^*\left(2\to \frac{32}{9}\right)\ll 1.\]

The main novelty in the proof of Rudnev and Shkredov \cite{RS} compared to prior work is improved estimates for the
additive energy of sets on the paraboloid $P$, see \cite[Lemma $7$]{RS}, where the additive energy of a set $A\subset P$, denoted by $E(A)$,
is defined as the number of quadruples $(a, b, c, d)\in A^4$ such that $a+b=c+d$. More precisely, to obtain new energy bounds, Rudnev and Shkredov \cite{RS} combined an argument on the distribution of angles in the Euclidean setting due to Pach and Sharir \cite{pach} and a point-plane incidence bound due to Rudnev \cite{plane}. 

However, when $d=3$ and $q\equiv 3\mod 4$, Conjecture \ref{L2ConjP} says that the bound $R_P^*(2\to 3)\ll 1$ is the optimal $L^2\to L^r$ estimate for paraboloids. Thus there is still a big gap between $32/9$ and $3$. It was mentioned in \cite{RS} that with the best additive energy estimate, one we can improve $32/9$ to $10/3$, which is of course still bigger than the expected exponent. 

We also note that there is a difference between Rudnev and Shkredov's method and that of Lewko in \cite{Le13}. More precisely, $E(A)$ can be reduced to the number of incidences between a certain point set and a certain line set in \cite{Le13}, and  to the number of right angles in one set in one lower dimension space in \cite{RS}. However, when $d$ is increasing, say $d=4k+3$ with $k\ge 1$ and $q\equiv 3\mod 4$, it seems not possible to generalize Rudnev and Shkredov's argument due to structures of subspaces in hyperplanes. 

The first listed author \cite{kk} adapted Lewko's paradigm in \cite{Le13} with some refinements and known energy bounds  showed that if $d=4k+3$, $k\in \mathbb{N}$, and $q\equiv 3\mod 4$, then we have 
\begin{equation}\label{lancuoi}R_P^*\left(2\to \frac{6d+10}{3d-1}+\epsilon\right)\ll 1,\end{equation}
for any $\epsilon>0$. 

If we want to improve this result, it is natural to think of improved energy bounds. In higher dimensions, to bound $E(A)$, the point-line incidence bound step in Lewko's method can be easily handled by using a result due to the third listed author in \cite{vinh}, but the main issue would be bounding the number of pairs of zero distance in one lower dimensional spaces. In general, compared to spheres of non-zero radii, the Fourier decay of the zero sphere is very weak in our dimensions (see \cite[Lemma 4.2]{hello}), which will imply to worse energy bounds. 

However, under assumptions that $d=4k+3$ and $q\equiv 3\mod 4$, it has been discovered recently by Iosevich, Lee, Shen, and the first and second listed authors \cite{pham} that the absolute value of the Fourier transform of the zero sphere is not good enough, but if we only care about the positive part, then it is much better than Fourier decay of other spheres. In particular, to get an upper bound for the energy, it is enough with the positive part. Thus, as a consequence, we can improve (\ref{lancuoi}) as follows. 
\begin{theorem}\label{thm:main3}
Let $P$ be the paraboloid in $\mathbb{F}_q^d$ with $d=4k+3$, $k\in \mathbb{N}$, and $q\equiv 3\mod 4.$ We have 
\[R_P^*\left(2\to \frac{2d+4}{d}\right)\ll 1.\]
\end{theorem}
We remark here that the Fourier transform of the zero sphere also gives us some clues to improve Theorem \ref{thm:main3} further, we will give a brief discussion in the proof of Theorem \ref{thm1}. A proof of Theorem \ref{thm:main3} will be provided in Section \ref{sec9}.  We also refer the interested reader to \cite{hickman} for a result in the setting of rings of integers modulo $N$. 
\subsubsection{\textbf{A connection with the Erd\H{o}s-Falconer distance conjecture}} One of the most important applications of restriction/extension conjectures is on the Kakeya problem. In Euclidean Fourier analysis, it has been shown that there is an explicit connection between the Kakeya and the restriction conjecture, namely, the restriction conjecture implies the Kakeya conjecture. We refer the interested reader to \cite{muoi, Ta04, bonhai} for discussions. 

In the finite field setting, the Kakeya problem was introduced by Wolff \cite{bonhai} in 1999, and was settled by Dvir in 2008 by using polynomial methods \cite{dvir}. 

Lewko \cite{european} proved that there also exists an explicit connection between the Kakeya conjecture and the restriction conjecture associated to the paraboloid in $\mathbb{F}_q^d$. More precisely, in odd dimensional spaces and $-1$ is a square, using the finite field Kakeya maximal operator estimates given by Ellenberg, Oberlin, and Tao in \cite{45}, he showed that 
\[R^*_P\left( \frac{2d+2}{d-1}\to \frac{2d+2}{d-1}-\delta_d\right)\ll 1,\]
for some $\delta_d>0$. Notice that Lewko's argument is rather different compared to that of Bourgain \cite{6} in 1991 in the Euclidean case. He also showed that the paraboloid restriction/extension conjecture in dimensions $2n+1$ and $-1$ is a square implies the Kakeya problem in dimensions $n+1$. 

In this paper, we will provide one more important application of the restriction/extension conjecture associated to the paraboloid $P$ on the Erd\H{o}s-Facolner distance problem, which is one of central open problems in finite discrete geometry. Before stating the connection, we start with a review as follows. 

For any two points $x=(x_1, \ldots, x_d)$ and $y=(y_1, \ldots, y_d)$ in $\mathbb{F}_q^d$, we define the distance between them by the following formula:
\[||x-y||=(x_1-y_1)^2+\cdots+(x_d-y_d)^2.\]
This function is not a metric, but it preserves the main properties of the Euclidean distance function, for example, it is invariant under translations and actions of elements in the orthogonal group. 

Given a set $A\subset \mathbb{F}_q^d$, we denote the set of distances in $A$ by $\Delta(A)$, i.e. 
\[\Delta(A):=\{||x-y||\colon x, y\in A\}.\]
The finite field analogues of the Erd\H{o}s distinct distances problem were first investigated by Bourgain, Katz, and Tao in $2003$ in the remarkable paper \cite{bkt}. In particular, assuming that $q\equiv 3 \mod{4}$ is a prime, they proved the following result in the plane. \\

\begin{theorem} [\cite{bkt}, Theorem 7.1]\label{BKTeq}
Suppose $q\equiv 3\mod 4$ is a prime. Let  $A$ be a point set in $\mathbb{F}_q^2$. If  $|A|=q^{\alpha}$ with $0<\alpha<2$, then we have 
\[|\Delta(A)|\gg |A|^{\frac{1}{2}+\epsilon},\]
for some positive $\epsilon=\epsilon(\alpha)>0$.
\end{theorem}
Iosevich and Rudnev \cite{IR06} observed that the conclusion of Theorem \ref{BKTeq} can not hold in general without the assumption that $q\equiv 3\mod{4}$ is a prime.  
Indeed, if $q=p^2 $ for some prime $p$, then we can take $A=\mathbb F_p \times \mathbb F_p.$ In addition, if $q\equiv 1 \mod{4},$ then we can consider $A=\{(t, it)\in \mathbb F_q^2: t\in \mathbb F_q\},$ where $i^2=-1$ for some $i\in \mathbb F_q$. Because of these reasons, Iosevich and Rudnev reformulated the problem in the spirit of the Falconer distance conjecture, one of the most important open conjectures in geometric measure theory, which says that if $A\subset\mathbb{R}^d$ is a compact set whose  Hausdorff dimension is strictly greater than $d/2$, then the distance set $\Delta(A)$ has positive Lebesgue measure. The precise statement of the problem over finite fields is as follows.

\begin{question}
How large does a subset $A$ of $\mathbb{F}_q^d$ need to be to guarantee that $\Delta(A)$ contains a positive proportion of the elements of $\mathbb{F}_q$? 
\end{question}

This question has been named as the Erd\H{o}s-Falconer distance problem over finite fields. Iosevich and Rudnev \cite{IR06} showed that if $|A|\ge 4q^{(d+1)/2}$, then $\Delta(A)=\mathbb{F}_q.$  Hart, Iosevich, Koh, and Rudnev \cite{hart} constructed concrete examples to demonstrate that  the exponent $(d+1)/2$ can not be improved for certain odd dimensions. More precisely, if $d\ge 3$ is odd except the case $d=4k-1$, $k\in \mathbb N$ and $q\equiv 3 \mod{4}$, then the exponent $(d+1)/2$ is sharp even we only wish to cover a positive proportion of all distances. For $d=2$, in order to get a positive proportion of all distances,  Chapman, Erdogan, Hart, Iosevich and Koh  \cite{CEHIK10} proved that the exponent $3/2$ can be decreased to $4/3$ directly in line with Wolff's result \cite{37} on the Falconer distance problem in $\mathbb{R}^2$. 

It is worth noting that the current best bound for the Falconer distance problem in $\mathbb{R}^2$ is due to Guth, Iosevich, Ou and Wang \cite{alex-fal}. In particular, they indicated that if $A\subset\mathbb{R}^2$ has Hausdorff dimension of at least $5/4$, then the distance set $\Delta(A)$ has positive Lebesgue measure. In higher dimensions, we refer readers to \cite{g-h} and references therein for more details. 

In the setting of finite fields, it has been conjectured that in order to have a positive proportion of all distances, the exponent $(d+1)/2$ can be reduced to $d/2$ when either $d\ge 2$ is even or $d=4k-1$ and $q\equiv 3\mod 4$, but no improvement has been made over the last ten years. We will state the conjecture formally as follows. 

\begin{conjecture}\label{con1-1}
Let $A$ be a set in $\mathbb{F}_q^d$ such that either $d\ge 2$ is even or $d=4k-1$ and $q\equiv 3\mod 4$. Suppose that $|A|\gg q^{d/2}$, then the distance set $\Delta(A)$ contains a positive proportion of all distances.
\end{conjecture}

We also have the following \textit{bi-variant} of this conjecture for two sets. 

\begin{conjecture}\label{con1-2}
Let $A, B$ be sets in $\mathbb{F}_q^d$ such that either $d\ge 2$ is even or $d=4k-1$ and $q\equiv 3\mod 4$. Suppose that $|A||B|\gg q^{d}$,  then  the distance set $\Delta(A, B)$ contains a positive proportion of all distances,  where $\Delta(A, B)=\{||a-b||\colon a\in A, b\in B\}$.
\end{conjecture}

We note that in the setting of Euclidean space, a recent work of Iosevich and Liu \cite{lex-liu} tells us that if we have two sets $A, B\subset \mathbb{R}^d$, then there exists a probability measure $\mu_B$ on $B$ such that for $\mu_B$-a.e $b\in B$, $\Delta^b(A)$ has positive measure whenever $dim_H(A)+\frac{d-1}{d+1}dim_H(B)>d$, where $\Delta^b(A)=\{||a-b||\colon a\in A\}$.

In the following, we are able to show that the paraboloid $L^p\to L^2$ restriction conjecture in dimensions $4k-1$ and $q\equiv 3\mod 4$ implies the exponent $\frac{d}{2}+\frac{d+2}{2(d+3)}$ on the Erd\H{o}s-Falconer distance problem in dimensions $4k-2$ and $q\equiv 3\mod 4$. 

\begin{theorem}\label{conn2}
Let $P$ be the paraboloid in $\mathbb{F}_q^{d+1}$. Assume that the $L^p\to L^2$ restriction estimate for $P\subset \mathbb F_q^{d+1}$ holds.  For $A\subset \mathbb F_q^d$ with $|A|\gg  \max\left\{ q^{d/2}, ~ q^{(dp-p+2)/(3p-2)}\right\}$, we have \[|\Delta(A)|\gg q.\]
\end{theorem}

When $(d+1)=4k-1$, $k\in \mathbb N,$ $q\equiv 3\mod 4$,  it is well-known from the restriction conjecture for $P \subset \mathbb F_q^{d+1}$ that $p=(2d+8)/(d+6)$ is the optimal $p$ value. Therefore, one can use Conjecture \ref{L2ConjP}  to obtain the exponent $\frac{d^2+4d+2}{2d+6}=\frac{d}{2}+\frac{d+2}{2(d+3)}$ for the distance problem which is better than the current threshold $\frac{d+1}{2}$. 

When $(d+1)$ is even, it is known that the optimal $p$ for the $L^p\to L^2$ restriction estimate for $P\subset \mathbb F_q^{d+1}$ is given by $p=(2d+6)/(d+5).$ Thus we recover the known exponent $\frac{d+1}{2}$.

A proof of Theorem \ref{conn2} will be given in Section \ref{sec11}.

\subsubsection{\textbf{Extension theorems for spheres}}
For $j\in \mathbb{F}_q$, let us recall the definition of the sphere of radius $j$ centered at the origin.
\begin{equation}\label{DefSj} S_j:=\{x\in \mathbb{F}_q^d\colon ||x||=j\}.\end{equation}
In this subsection, we will present new extension theorems for spheres in odd dimensional spaces. 

It is well-known that in the Euclidean space, the extension theorems for paraboloids and spheres are the same, but in the setting of finite fields, the problems are completely different.  Compared to the case of paraboloids or cones, it has been believed that the spherical extension problem is much harder to understand, since the Fourier transform of the sphere is closely related to the Kloosterman sum whose explicit form is not known. In the paraboloid case,  we will see that (proof of Theorem \ref{thm:main3}), there is a connection between the $L^2\to L^r$ estimate and the additive energy bound. In the finite field setting, such a connection was initially given by Mockenhaupt and Tao \cite{MT04}, and a more precise relation between them was found by Lewko \cite{Le13}.  However, it seems that there is no such link for the case of spheres. 

If the radius of the sphere $S_j$ is not zero, then the Stein-Tomas method, which relies on decay properties of $d\sigma^\vee$, gives the following estimate 
$$R_{S_j}^*\left(2\to \frac{2d+2}{d-1}\right)\ll 1,$$
see Theorem $1$ in \cite{06} for a detailed proof. 

Interpolating this result and the trivial bound $R_{S_j}^*(1\to \infty)\ll 1$, we obtain 
\[R_{S_j}^*\left(\frac{4d-4}{3d-5}\to 4\right)\ll 1,\]
which we refer as the Stein-Tomas exponent toward $L^p\to L^4.$

In even $d\ge 4$, Iosevich and the first listed author \cite{I-K} showed  that $R_{S_j}^*(\frac{12d-8}{9d-12}\to 4)\ll 1$, which is better than the Stein-Tomas exponent toward $L^p\to L^4$. In a recent work, Iosevich, Lee, Shen, and the first two listed authors \cite{pham} provided an optimal $L^p\to L^4$ estimate in even dimensions, namely, they proved that 
\[R^*_{S_j}\left(\frac{4d}{3d-2}\to 4\right)\ll 1,\]
for any sphere $S_j$ of non-zero radius. Compared to the work \cite{I-K}, the key ingredient in \cite{pham} is the sharp additive energy bound of sets on spheres in even dimensions. We also remark that combining the method in \cite{pham} and the additive energy bounds in \cite{hello} yields the same result for the paraboloid $P$ in even dimensions, namely, 
\[R^*_{P}\left(\frac{4d}{3d-2}\to 4\right)\ll 1.\]

In odd dimensional spaces, over the last ten years, it has been believed in \cite{I-K} that the Stein-Tomas exponent toward $L^p\to L^4$ can not be improved in general. Indeed, they showed that if $q\equiv 1\mod 4$ and $d$ is odd, then the unit sphere contains an affine subspace of dimension $(d-1)/2$. This construction can be easily derived, for example, we may assume that $d=5$ and $q=1\mod 4,$ and we define 
\[H:=(0, 0, 0, 0, 1)+Span\left((1, i, 0, 0, 0), (0, 0, 1, i, 0)\right).\]
It is clear that $H$ is contained in $S_1$ and $|H|=q^{(d-1)/2}$. 
Now applying \eqref{NecExt} yields that if $r=4$, then the threshold $p=\frac{4d-4}{3d-5}$ is best possible. 

It is surprising that when we take into account the radii of spheres, i.e. square or non-square,  the Stein-Tomas exponent toward $L^p\to L^4$ can be considerably improved (see Conjecture \ref{conj1.3} and our results below). 

Let $g$ be a primitive element of $\mathbb{F}_q$, i.e. a generator of $\mathbb{F}_q^*$. We have the following extension theorems for $S_g$ which improve the Stein-Tomas exponent toward $L^p\to L^4$ in odd dimensions.

\begin{theorem}\label{thm:main3'}
Let $g$ be a primitive element in $\mathbb{F}_q.$ If $d=4k+1,$ then we have
\[R^*_{S_g}\left(\frac{4d}{3d-2}\to 4\right)\ll 1.\]
\end{theorem}

\begin{theorem}\label{thm:main3''}
Let $g$ be a primitive element in $\mathbb{F}_q.$ Suppose $d=4k-1$ and $q\equiv 1\mod 4$. Then we have
\[R^*_{S_g}\left(\frac{4d}{3d-2}\to 4\right)\ll 1.\]
\end{theorem}

The proofs of Theorems \ref{thm:main3'} and \ref{thm:main3''} will be provided in Section \ref{sec10}. 

The novelty in our proofs of Theorems \ref{thm:main3'} and \ref{thm:main3''} is new estimates for the additive energy of sets on spheres of primitive radii, which goes beyond the capacity of methods in \cite{hello, Le13, RS}. Our main tool is  \textit{the first association scheme graph} introduced by Bannai, Osamu, Hajime Tanaka \cite{BSH} in 2004. This is the first time \textit{the first association scheme graph} is used in this topic. The role of this graph is similar to the Fourier transform of the zero sphere in the paraboloid case. More precisely, it will help us to detect pairs of zero distance in a given set

As a consequence, we will be able to show that for $A\subset S_g,$ 
\[E(A)\ll \frac{|A|^3}{q}+q^{\frac{d-2}{2}}|A|^2.\]
If $g$ is a square, i.e. non-primitive, then this bound is impossible in general. For instance, in the above construction in $\mathbb{F}_q^5$, we have $H\subset S_1$ and $E(H)\sim |H|^3$. The main difference between these cases comes from the problem of estimating number of pairs $(a, b)\in A^2$ such that $||a-b||=0$. This is a difficult problem since spheres might contain many isotropic lines, i.e. any two points on those lines have zero norm. For example, if there are two points $a, c\in S_g$ such that $||a-c||=0$, then the line $x=c+t(a-c)$ with $t\in \mathbb{F}_q$ is contained fully in $S_g$. In the above construction, the number of such pairs in $H$ is $|H|^2$. If a sphere has primitive radius, then we can take  advantages of techniques in algebraic combinatorics, namely, \textit{association action scheme} to overcome the difficulties.

We note here that, in general, for $A\subset S_j, j\ne 0,$ we always have 
\[E(A)\ll \frac{|A|^3}{q}+q^{\frac{d-1}{2}}|A|^2.\]
The upper bound can be attained by the above construction. 


Unlike the paraboloid case, statements of reasonable conjectures on extension estimates for spheres have not appeared in the previous literature. In this subsection, we will deduce more accurate necessary conditions for $L^p \to L^4$ extension estimates for spheres $S_j.$ To this end, we will invoke the following lemma which can be taken from  \cite[P.79]{Gr02} or \cite[Theorem 1]{AM16}.

\begin{lemma}\label{LemEq} Let $Q(x)=x_1^2+x_2^2+\cdots+x_d^2\in \mathbb F_q[x]$ where $x=(x_1,\ldots,x_d).$ Suppose that $\lambda$ is a fixed 
non-square number in $\mathbb F_q$ and $\eta$ denotes the quadratic character of $\mathbb F_q^*.$ Then one of the following holds:
\begin{enumerate}
\item
If $d\ge 2$ is even, then $Q(x)$ is equivalent to the form
\begin{equation*} x_1^2-x_2^2+ \cdots + x_{d-3}^2-x_{d-2}^2+ x_{d-1}^2 - \alpha x_d^2,\end{equation*}
where $\alpha$ is the element in $\mathbb F_q^*$ such that 
$\alpha\in \{1, \lambda\}$ and $1=\eta((-1)^{d/2}) \eta(\alpha).$

\item If $d\ge 3$ is odd, then $Q(x)$ is equivalent to the form
\begin{equation*}
x_1^2-x_2^2+ \cdots +x_{d-2}^2- x_{d-1}^2 + \alpha x_d^2,
\end{equation*}
where  $\alpha$ is the element in $\mathbb F_q^*$ satisfying that $\alpha \in \{1, \lambda\}$ and $1=\eta((-1)^{(d-1)/2}) \eta(\alpha).$
\end{enumerate}
\end{lemma}


From Lemma \ref{LemEq}, we are able to detect  an affine subspace lying on the sphere $S_j=\{x\in \mathbb F_q^d: x_1^2+\cdots+x_d^2=j\}.$

\begin{lemma}\label{sizeASS} Let $S_j$ be the sphere in $\mathbb F_q^d$ with $j\ne 0.$ Then the following statements hold:
\begin{enumerate} \item If $d\ge 2$ is even , then $S_j$ contains an affine subspace $H$ with $|H|=q^{(d-2)/2}.$
\item If $d=4k+1$, $k\in \mathbb N,$ and $j$ is not square, then $S_j$ contains an affine subspace $H$ with $|H|=q^{(d-3)/2}.$
\item If $d=4k+1$, $k\in \mathbb N,$ and  $j$ is square, then $S_j$ contains an affine subspace $H$ with $|H|=q^{(d-1)/2}.$
\item If $d=4k-1$, $k\in \mathbb N,$ and $-j$ is not square, then  $S_j$ contains an affine subspace $H$ with $|H|=q^{(d-3)/2}.$
\item If $d=4k-1$,  $k\in \mathbb N,$  and  $-j$ is  square, then  $S_j$ contains an affine subspace $H$ with $|H|=q^{(d-1)/2}.$
\end{enumerate}
\end{lemma}
\begin{proof}  
To prove the first part of the lemma, we  note from the first part of Lemma \ref{LemEq} that  $S_j$ is equivalent to the variety
$$ \widetilde{S_j}:=\{x\in \mathbb F_q^d: x_1^2-x_2^2+ \cdots + x_{d-3}^2-x_{d-2}^2+ x_{d-1}^2 - \alpha x_d^2=j\},$$
where $\alpha$ is a non-zero element of $\mathbb F_q$. Let $(a,b)\in \mathbb F_q^2$ such that $a^2-\alpha b^2=j.$ Then $\widetilde{S_j}$ clearly contains a $(d-2)/2$-dimensional affine subspace with the form
$$\left\{(t_1, t_1, t_2, t_2, \ldots, t_{\frac{d-2}{2}}, t_{\frac{d-2}{2}}, a, b)\in \mathbb F_q^d: t_i\in \mathbb F_q,  i=1,2,\ldots, (d-2)/2 \right\}.$$ Hence, $S_j$ also contains a $(d-2)/2$-dimensional affine subspace. This completes the proof of the first part of the lemma. 

To prove other conclusions, we observe from the second part of Lemma \ref{LemEq} that
$S_j$ is equivalent to the variety
\begin{equation}\label{oddH}
S_j(\alpha):=\{x\in \mathbb F_q^d:x_1^2-x_2^2+ \cdots +x_{d-2}^2- x_{d-1}^2 + \alpha x_d^2=j\},
\end{equation}
where  $\alpha$ is the element in $\mathbb F_q^*$ satisfying that $\alpha \in \{1, \lambda\}$ and $1=\eta((-1)^{(d-1)/2}) \eta(\alpha).$

\textbf{Case 1:} Assume that there exists $a\in \mathbb F_q^*$ such that $\alpha a^2=j.$ This case is equivalent to 
$\eta(\alpha j)=1,$ where $\eta$ denotes the quadratic character of $\mathbb F_q^*.$  By a direct computation, one can check that 
each assumption of the third part and the fifth part of Lemma \ref{LemEq} satisfies the condition that $\eta(\alpha j)=1.$
In addition, we notice that if $\alpha a^2=j$ for some $a\in \mathbb F_q^*,$ then  $S_j(\alpha)$ contains a $(d-1)/2$-dimensional affine subspace with the form
$$\left\{(t_1, t_1, t_2, t_2, \ldots, t_{\frac{d-1}{2}}, t_{\frac{d-1}{2}}, a)\in \mathbb F_q^d: t_i\in \mathbb F_q,  i=1,2,\ldots, (d-1)/2 \right\}.$$
Hence, a $(d-1)/2$-dimensional affine subspace $H$ is contained in the sphere $S_j$ under the assumption of the third part or the fifth part of Lemma \ref{sizeASS}. This completes proofs of the third and fifth parts of Lemma \ref{LemEq}.

\textbf{Case 2:} Assume that there is no $a\in \mathbb F_q^*$ such that $\alpha a^2=j,$ which is equivalent to the case when $\eta(\alpha j)=-1.$ It is not hard to see that each assumption of the second and  the fourth parts of Lemma \ref{sizeASS} satisfies  this case. Let $(a,b,c)\in \mathbb F_q^3$ such that $a^2-b^2+\alpha c^2=j.$  It follows from \eqref{oddH} that $S_j(\alpha)$ contains a $(d-3)/2$-dimensional affine subspace with the form
$$ \left\{(t_1, t_1, t_2, t_2, \ldots, t_{\frac{d-3}{2}}, t_{\frac{d-3}{2}}, a, b, c)\in \mathbb F_q^d: t_i\in \mathbb F_q,  i=1,2,\ldots, (d-3)/2 \right\},$$
which completes proofs of the second and fourth parts of Lemma \ref{sizeASS}.
\end{proof}

%

Combining the necessary conditions in \eqref{NecExt} (or \eqref{Necovex}) and Lemmea \ref{sizeASS},   one may conjecture all exponents $1\le p,r\le \infty$ such that $R^*_{S_j}(p\to r)\ll 1.$ 
In particular, we can state the following conjecture for the sharp bound $R^*_{S_j}(p\to 4)\ll 1$ in odd dimensions. \\

\begin{conjecture}\label{conj1.3} Let $S_j$ be the sphere with non-zero radius in $\mathbb F_q^d.$   The following statements hold.
\begin{enumerate}
\item If $d=4k+1$, $k\in \mathbb N,$ and $j$ is not square, then the bound $R^*_{S_j}\left(\frac{4d+4}{3d+1} \to 4\right)\ll 1$ gives the sharp $L^p\to L^4$ estimate.
\item If $d=4k-1$, $k\in \mathbb N,$ $q\equiv 1 \mod{4}$, and $j$ is not square, then the bound $R^*_{S_j}\left(\frac{4d+4}{3d+1} \to 4\right)\ll 1$ gives the sharp $L^p\to L^4$ estimate.
\item If $d=4k-1$, $k\in \mathbb N,$ $q\equiv 3 \mod{4},$ and $j$ is square, then  the bound $R^*_{S_j}\left(\frac{4d+4}{3d+1} \to 4\right)\ll 1$ gives the sharp $L^p\to L^4$ estimate.
\end{enumerate} 
\end{conjecture}

Notice that the first part of the above conjecture is based on the second part of Lemma \ref{sizeASS}. On the other hand, the second and third parts of the above conjecture come from the fourth part of Lemma \ref{sizeASS}.

\begin{remark} We have an interesting observation when comparing Proposition \ref{lemP} for paraboloids to   the first and second statements of Conjecture \ref{conj1.3} for spheres in odd dimensions.
Ignoring the assumption on the radius of spheres,  they have the same assumptions but the expected conclusions are different. Conjecture \ref{conj1.3} says that the sharp Stein-Tomas exponent toward $L^p\to L^4$ for paraboloids can be improved significantly for the corresponding spheres in odd dimensions. It is clear that Theorems \ref{thm:main3'} and \ref{thm:main3''} are partial evidences.
\end{remark} 

As we mentioned before, there is no known bridge between $L^2\to L^r$ estimates for spheres and the additive energy bound. This leads to a challenge to improve the Stein-Tomas result. The conjecture on $L^2\to L^r$ extension estimate for the spheres in even dimensions is stated as follows. 

\begin{conjecture}\label{duoc}
For $d\ge 2$ even, let $S_j$ be the sphere of radius $j\ne 0$ centered at the origin in $\mathbb{F}_q^d$. We have the following $L^2\to L^r$ estimate 
\[R_{S_j}^*\left( 2\to \frac{2d+4}{d}\right)\ll 1.\]
\end{conjecture}

Notice that the case $d=2$ of Conjecture \ref{duoc} has been proved by Chapman, Erdogan, Hart, Iosevich, and Koh in \cite{CEHIK10}. However, in higher dimensional spaces, their method gives us the estimate $R_{S_j}^*\left( 2\to 4\right)\ll 1,$
which is very weak compared to the Stein-Tomas result.

It is worth noting that it has been shown in \cite{CEHIK10} that the spherical $L^2\to L^r$ extension conjecture in even dimensions will imply the exponent $\frac{d}{2}+\frac{d}{2(d+1)}$ on the Erd\H{o}s-Falconer distance problem. However, what we know about the spherical $L^2\to L^r$ extension conjecture is very limited compared to the paraboloid case.

\subsection{\textbf{A step towards the distance Conjecture \ref{con1-2}}}
Theorem \ref{conn2} tells us that the restriction conjecture for the paraboloid $P$ in $\mathbb{F}_q^{d+1}$ with $d=4k-2$, $k\in \mathbb N,$ $q\equiv 3\mod 4$, implies the exponent $\frac{d}{2}+\frac{d+2}{2(d+3)}$ on the Erd\H{o}s-Falconer distance problem in $\mathbb{F}_q^d$. Conjectures \ref{con1-1} and \ref{con1-2} suggest that the right exponent should be $d/2$. 

In this section, we will make a step towards Conjecture \ref{con1-2} by showing that it holds when we assume that one set lies on a variety (a sphere or a paraboloid). In fact, our statements will be stronger in certain dimensions when the variety is a sphere. We also note that our coming results are sharp in odd dimensional spaces.

\subsubsection{\textbf{Distances between a set on a paraboloid and an arbitrary set}}
Recall that the paraboloid $P$ in $\mathbb{F}_q^d$ is defined as follows. 
\[P:=\{(x_1, x_2, \ldots, x_{d-1}, x_d)\colon x_d=x_1^2+\cdots+x_{d-1}^2, (x_1, \ldots, x_{d-1})\in \mathbb{F}_q^{d-1}\}.\]
We have the following result on distances between a set on a paraboloid and an arbitrary set in $\mathbb{F}_q^d$.

\begin{theorem} \label{mainRP}  Let $A$ be a set on $P$ in $\mathbb{F}_q^d$, and $B$ be an arbitrary set in $\mathbb{F}_q^d$. The following consequences hold:\\
\begin{itemize}
\item For $d=4k-1$, $k\in \mathbb N,~q\equiv 3 ~\mbox{mod}~4$, if $|A||B|\ge 4q^d$, then we have
$$ |\Delta(A,B)|\ge \frac{q}{3}.$$
\item For even $d\ge 4,$ if  $|A||B|\ge 16q^d$ and  $|A|\not\in (q^{(d-1)/2}, q^{d/2})$, then we have
$$ |\Delta(A,B)|\ge \frac{q}{144}.$$
\end{itemize}
\end{theorem}

The following is a direct corollary of Theorem \ref{mainRP}.

\begin{corollary}
Let $A$ be a set on $P$ in $\mathbb{F}_q^d$. Suppose that either $d=4k-1$ and $q\equiv 3\mod 4$ or $d\ge 4$ is even, and $|A|\gg q^{d/2}$, then we have $|\Delta(A)|\gg q$.
\end{corollary}

A proof of Theorem \ref{mainRP} will be given in Section \ref{sec7}.
\subsubsection{\textbf{Distances between a set on a sphere and an arbitrary set}}
For $j\in \mathbb{F}_q$, recall that the sphere $S_j$ of radius $j$ centered at the origin in $\mathbb{F}_q^d$ is defined by 
\[S_j:=\{x=(x_1, \ldots, x_d)\in \mathbb{F}_q^d\colon ||x||=x_1^2+\cdots x_d^2=j\}.\]

In our next theorem, we consider distances between one set lies on $S_j$ with $j\ne 0$ and an arbitrary set in odd dimensional spaces.

\begin{theorem}\label{mainRSO}
Let $A\subset S_j$ with $j\ne 0$, and $B\subset \mathbb F_q^d.$ Then the following two statements hold:
\begin{itemize}
\item  Let $j$ be  a square number of $\mathbb F_q^*$. For $d=4k-1$, $k\in \mathbb N,$  $q\equiv 3 \mod 4,$ if $|A||B|\ge 4q^d$, then we have
$$ |\Delta(A,B)|\ge \frac{q}{4}.$$
\item   Let $j$ be a non-square number of $\mathbb F_q^*$. For either $d=4k+1$, $k\in \mathbb N$, or $d=4k-1$ and $q\equiv 1 \mod 4,$ if $|A||B|\ge 4q^d$,  then we have
$$ |\Delta(A,B)|\ge \frac{q}{4}.$$
\end{itemize}
\end{theorem}

In even dimensions, we have

\begin{theorem}\label{mainRSE} Let $A$ be a set on $S_j$ with $j\ne 0$, and $B$ be an arbitrary set in $\mathbb{F}_q^d$. 
For even $d\ge 4,$ if $|A||B|\ge 16q^d$ and $|A|\not\in (q^{(d-1)/2}, q^{d/2})$, then we have
$$ |\Delta(A,B)|\ge \frac{q}{144}.$$
\end{theorem}

While our results for spheres are similar to those of paraboloids, the proofs are much more complicated since the Fourier decay of spheres is associated with the Kloosterman sum whose explicit form is not known.

As consequences of Theorems \ref{mainRSO} and \ref{mainRSE}, if $A=B\subset S_j$ with $j\ne 0$, then we recover main results in \cite{hart} in even dimensions and in odd dimensions with some additional conditions.  When $A=B\subset S_j$, we can reduce the distance problem to the dot product problem, which is much easier with a geometric property that any line contains at most two points from $A$. However, in the form of Theorems \ref{mainRSO} and \ref{mainRSE}, one can check that such reduction does not work. 

It is more interesting when we replace the sphere $S_j$ by the sphere of radius zero $S_0$. Indeed, it is clear that the distance between two points $x, y\in S_0$ is 
\[||x-y||=-2x_1y_1-\cdots-2x_dy_d=-2x\cdot y.\]
Therefore, the distance problem is equivalent with the dot product problem. Nevertheless, we now face a bigger problem that the zero sphere $S_0$ contains many lines. This means that for any subset $A\subset S_0$, there might exist lines with $q$ points from $A$. Because of this reason, the method in \cite{hart} only gives us the exponent $(d+1)/2$. In the next theorem, using the unusual good Fourier decay of the zero sphere in dimensions $4k+2$ and $q\equiv 3\mod 4$, we are able to conquer the difficulties and to obtain the same result as in Theorem \ref{mainRSE}.

\begin{theorem}\label{mainRZS} Let $S_0$ be the sphere with zero radius centered at the origin in $\mathbb F_q^d.$ Let $A$ be a set on $S_0$ and $B$ be an arbitrary  set in $\mathbb{F}_q^d$. For $d=4k+2$, $k\in \mathbb N$,  $q\equiv 3\mod 4$, if $|A||B|\ge 16q^d$ and $|A|\not\in (q^{(d-1)/2}, q^{d/2})$, then we have
$$ |\Delta(A,B)|\ge \frac{q}{144}.$$
\end{theorem} 

\begin{corollary}
Let $A$ be a set on $S_0$ in $\mathbb{F}_q^d$. Suppose that $d=4k+2$, $k\in \mathbb N$, $q\equiv 3\mod 4$, and $|A|\gg q^{d/2}$, then we have $|\Delta(A)|\gg q$. 
\end{corollary}
We will prove Theorem \ref{mainRSO}, \ref{mainRSE}, and \ref{mainRZS} in Section \ref{sec7}.
\subsubsection{\textbf{Sharpness of distance results}}

Notice that in odd dimensional spaces our distance results are sharp. More precisely,  for any $\epsilon>0$,  there exist sets $A$ on $P$ or $S_j$ and $B\subset \mathbb{F}_q^d$ with $|A||B|\gg q^{d-\epsilon}$ and $|\Delta(A, B)|=o(q)$ (see Lemmas \ref{para-3} and \ref{sphere-c-1}). In even dimensional spaces,  our condition $|A||B|\gg q^d$ is sharp (see Lemmas \ref{sphere-c-3} and \ref{sphere-c-3-2}). Moreover, in even dimensions,  if there exist a set $A$ on $P$ or $S_j$ with $q^{(d-1)/2}< |A|< q^{d/2}$ and $B\subset \mathbb{F}_q^d$ satisfy $|A||B|\ge q^d$ and $|\Delta(A, B)|=o(q)$, then the results would be sharp, but such a construction is not known.

The rest of this paper is organized as follows. In Section $2$, we derive energy bounds which are the most important parts in proofs of extension theorems. Proofs of the extension theorem associated to the paraboloid $P$ and the connection with the Erd\H{o}s-Falconer distance problem will be provided in Sections $3$ and $4$, respectively. Section $5$ is devoted for proofs of extension theorems associated to spheres. We prove theorems on distances between two sets in Section $6$. The sharpness of distance results will be discussed in Section $7$.
\section{Energy bounds}\label{sec8}
Let us first recall the definition of the additive energy of a set on a variety. Let $V$ be a variety in $\mathbb{F}_q^d$. For $A\subset V$, the additive energy of $A$, denoted by $E(A)$, is defined by 
\[E(A)=|\{(a, b, c, d)\in A^4\colon a+b=c+d\}|.\]
In this section, we bound $E(A)$ when $A$ is a set on a sphere or the paraboloid $P$. Energy bounds play the crucial role in our proofs of extension theorems. 
\subsection{Energy for paraboloids}
\begin{theorem}\label{thm1}
Let $A$ be a set on a paraboloid $P\subset \mathbb{F}_q^{d}$ with $d=4k+3$, $k\in \mathbb{N}$ and $q\equiv 3\mod 4$. We have the following estimate
\[E(A)\ll \frac{|A|^3}{q}+q^{\frac{d-2}{2}}|A|^2.\]
\end{theorem}
To prove Theorem \ref{thm1}, we first prove the following lemma.

\begin{lemma}\label{thm3}
Let $A\subset \mathbb{F}_q^{d}$ with $d=4k+2$, $k\in \mathbb{N}$, and $q\equiv 3\mod 4$. Let $N(A)$ be the number of pairs $(a, b)\in A^2$ such that $||a-b||=0$. Then we have 
\[N(A)\ll \frac{|A|^2}{q}+q^{\frac{d-2}{2}}|A|.\]
\end{lemma}
\begin{proof}
To prove this lemma, we recall the following result from \cite[Lemma $3.2$]{pham}. 
\begin{lemma}\label{lmx9}
Let $S_0$ be the sphere centered at the origin of radius $0$ in $\mathbb{F}_q^d$. Suppose that $d=4k+2$ and $q\equiv 3\mod 4.$ Then, for any $\alpha\in \mathbb{F}_q^d$, we have 
\[\widehat{S_0} (\alpha)=\frac{1}{q}\cdot \delta_0(\alpha)-q^{\frac{-(d+2)}{2}}\sum_{r\ne 0}\chi(r||\alpha||),\]
where  $\delta_0(\alpha)=1$ if $\alpha=0$ and zero otherwise. 
\end{lemma}
We have 
$$N(A)=\sum_{a, b\in A: a-b\in S_0}1 =\sum_{a, b\in  \mathbb F_q^d} A(a) A(b) S_0(a-b).$$
Applying the Fourier inversion formula to the indicator function $S_0(a-b)$,  we obtain
\begin{align}\label{GS} N(A) &= \sum_{a, b\in \mathbb F_q^d} A(a) A(b) \sum_{m\in \mathbb F_q^d}  \widehat{S_0}(m) ~\chi(m\cdot (a-b))\nonumber\\
&=q^{2d}\sum_{m\in \mathbb F_q^d} |\widehat{A}(m)|^2 \widehat{S_0}(m).
\end{align}
It follows from Lemma \ref{lmx9} that 
$$N(A)\le q^{2d}\left(\sum_{m\in \mathbb F_q^d} |\widehat{A}(m)|^2 q^{-1} \delta_0(m) -q^{\frac{-(d+2)}{2}} \sum_{m\in \mathbb F_q^d}  |\widehat{A}(m)|^2 \sum_{r\ne 0} \chi(r\|m\|)\right).$$
Using the orthogonality property of $\chi$, we get the following  
$$N(A)\le   \frac{|A|^2}{q}  -q^{2d}\cdot q^{\frac{-(d+2)}{2}} \cdot (q-1) \sum_{\|m\|=0}  |\widehat{A}(m)|^2  + q^{2d}\cdot q^{\frac{-(d+2)}{2}} \sum_{\|m\|\ne 0}  |\widehat{A}(m)|^2$$
$$\le \frac{|A|^2}{q}+ q^{2d}\cdot q^{\frac{-(d+2)}{2}} \sum_{m\in \mathbb F_q^d}|\widehat{A}(m)|^2.$$
On the other hand, we also have $\sum_{m\in \mathbb F_q^d}  |\widehat{A}(m)|^2 =q^{-d}|A|$, which implies that
\begin{align*}
N(A)&\ll \frac{|A|^2}{q}+q^{\frac{d-2}{2}}|A|.\end{align*}
This concludes the proof of the lemma.
\end{proof}
%

We are now ready to prove Theorem \ref{thm1}.

\begin{proof}[Proof of Theorem \ref{thm1}]

We start with the following observation:\\
Given $a=(\underline{a}, \underline{a}\cdot \underline{a}), b=(\underline{b}, \underline{b}\cdot \underline{b}), c=(\underline{c}, \underline{c}\cdot \underline{c}), d=(\underline{d}, \underline{d}\cdot \underline{d})\in P,$ if $a+b=c+d$, then we have 
\[(\underline{a}-\underline{d})\cdot (\underline{b}-\underline{d})=0,\]
i.e. we have a \textit{right angle} at $\underline{d}$. This is equivalent to
\[(\underline{a}-\underline{b}, ||\underline{a}-\underline{b}||)-(\underline{d}-\underline{b}, ||\underline{d}-\underline{b}||)\in P.\]
We now partition $E(A)$ into a sum of $E_1$ and $E_2$, where $E_1$ is the number of tuples $(a, b, c, d)\in A^4$ such that $a+b=c+d$ and either $||\underline{a}-\underline{b}||=0$ or $||\underline{d}-\underline{b}||=0$, $E_2$ is the number of tuples $(a, b, c, d)\in A^4$ such that $a+b=c+d$ and $||\underline{a}-\underline{b}||\ne 0$, $||\underline{d}-\underline{b}||\ne 0$. 

Since $d=4k+3$, we have $d-1=4k+2$. It follows from Lemma \ref{thm3} that 
\[E_1\ll \frac{|A|^3}{q}+q^{\frac{d-3}{2}}|A|^2.\]

To bound $E_2$, we do as follows. 

For a fixed $b\in A$, we now count the number of pairs $(a, d)\in A^2$ such that 
\begin{equation}\label{eq2xx}
(\underline{a}-\underline{b}, ||\underline{a}-\underline{b}||)-(\underline{d}-\underline{b}, ||\underline{d}-\underline{b}||)\in P.\end{equation}
Define 
\[A':=\{(\underline{x}-\underline{b}, ||\underline{x}-\underline{b}|| )\colon (\underline{x}, ||\underline{x}||)\in A\}\subset P.\]
It is not hard to check that the number of pairs $(a, d)\in A^2$ satisfying (\ref{eq2xx}) is equal to the number of pairs $(a', d')\in A'^2$ such that $a'-d'\in P$. We also observe that given $x, y\in P$, if $x-y\in P$, then we have $\underline{x}\cdot \underline{y}=||\underline{y}||$.

It is clear that the equation $\underline{a}\cdot \underline{b}=||\underline{b}||$ is equivalent with an incidence between the point $\underline{a}$ in $\mathbb{F}_q^{d-1}$ and the hyperplane $x\cdot \underline{b}=||\underline{b}||$ in $\mathbb{F}_q^{d-1}$. Thus, a point-hyperplane incidence bound in \cite{vinh} tells us that the number of pairs $(a', d')\in A'^2$ such that $a'-d'\in P$ is at most 
\[\frac{|A|^2}{q}+q^{\frac{d-2}{2}}|A|.\]
Summing over all $b\in A$, we obtain 
\[E_2\ll \frac{|A|^3}{q}+q^{\frac{d-2}{2}}|A|^2.\]
In other words, 
\[E(A)=E_1+E_2\ll \frac{|A|^3}{q}+q^{\frac{d-2}{2}}|A|^2.\]
We would like to add a brief discussion here that the way we are using the point-hyperplane incidence bound here is very ``rough" in the sense that each hyperplane is of a very special form $x\cdot \underline{b}=||\underline{b}||$. We conjecture that the term $q^{\frac{d-2}{2}}|A|^2$ in $E_2$ should be improved to $q^{\frac{d-3}{2}}|A|^2$ as in $E_1$. This would give us the sharp energy, which leads to a big improvement on $L^2\to L^r$ paraboloid extension estimates. To see why the term $q^{\frac{d-3}{2}}|A|^2$ is sharp, we can construct the following example. One can use Lemma \ref{oth-tho} (below) to obtain a set $V$ in $\mathbb{F}_q^{d-3}$ such that $|V|=q^{\frac{d-3}{2}}$ and $v_i\cdot v_j=0$ for all $v_i, v_j\in V$ ($v_i$ and $v_j$ might be the same). Set $A=V\times \{0\}\times \{0\}\times \{0\}$. We have $A$ is a subset on $P$ and for all $a, b, c\in A$, we have $a+b-c\in A$. This gives us $E(A)=|A|^3=q^{\frac{d-3}{2}}|A|^2$.  
\end{proof}
\subsection{Energy for spheres in odd dimensions}\label{sh-e}
We start this section by recalling the definition of primitive elements.

\begin{definition}
Let $\mathbb{F}_q$ be a finite field of order $q$ such that $q$ is an odd prime power. An element $g\in \mathbb{F}_q$ is called a primitive element if $g$ is a generator of the group $\mathbb{F}_q^*$. 
\end{definition}

Let $g$ be a primitive element in $\mathbb{F}_q$, and $S_g$ be the sphere of radius $g$ centered at the origin in $\mathbb{F}_q^d$.  We have the following theorem on energy of a set on $S_g$.

\begin{theorem}\label{energy-sphere}
Suppose either $d=4k-1$ and $q\equiv 1\mod 4$ or $d=4k+1$. For $A\subset S_g$, we have
\[E(A)\ll \frac{|A|^3}{q}+q^{\frac{d-2}{2}}|A|^2.\]
\end{theorem}
\subsubsection{Proof of Theorem \ref{energy-sphere}}
We first recall the definition of association schemes in \cite{banai}. 
\paragraph{Association schemes:}
Let $X$ be a set of size $n$, and let $R_i$, $0\le i\le k$, for some integer $k$, be subsets of $X\times X$ with the following properties:
\begin{enumerate}
\item $R_0=\{(x, x)\colon x\in X\}$.
\item $X\times X=R_0\cup\ldots\cup R_k$, $R_i\cap R_j=\emptyset$ for any $i\ne j$. 
\item $R_i^t=R_{i'}$, for some $i'\in \{0, 1, \ldots, k\}$, where $R_i^t=\{(x, y)\colon (y, x)\in R_i\}$.
\item For $m, n, v\in \{0, 1, \ldots, k\}$, the number of $z\in X$ such that $(x, z)\in R_m$ and $(z, y)\in R_n$ is constant whenever $(x, y)\in R_v$. We denote this constant by $p^v_{mn}$. 
\item $p^v_{mn}=p^v_{nm}$ for all $m, n, v$. 
\end{enumerate}
Such a configuration, denoted by $\mathcal{X}=(X, (R_i)_{0\le i\le k})$, is called a commutative association scheme of class $k$ on $X$. The non-negative integers $p^v_{mn}$ are called the intersection numbers of $\mathcal{X}$. If the property ($5$) does not hold, then the configuration $\mathcal{X}$ is called a non-commutative association scheme. 

Over the last decades, the association schemes have been  intensively used to study Ramanujan graphs (see \cite{BHS, bst} and references therein). 

In this paper, we will follow a construction in \cite{BSH} to prove Theorem \ref{energy-sphere}. 

Suppose $d=2m+1$. Let 
\[Q(x)=2(x_1x_{m+1}+x_2x_{m+2}+\cdots+x_mx_{2m})+x_{2m+1}^2\]
be a quadratic form in $\mathbb{F}_q^d$. For each non-zero element $x\in \mathbb{F}_q^d$, we write $[x]$ for the $1$-dimensional subspace of $\mathbb{F}_q^d$ containing $x$. Let $V$ be the set of all non-square type non-isotropic $1$-dimensional subspaces of $\mathbb{F}_q^d$ corresponding to $Q(x)$. We have $|V|=q^m(q^m-1)/2$. 

One can check that the simple orthogonal group $O(\mathbb{F}_q^d)$ acts transitively on $V$, and we obtain a symmetric association scheme of class $(q+1)/2$ from this action. We denote this scheme by $\mathcal{X}(V, (R_i)_{0\le i\le (q+1)/2})$, where $R_i$ are defined as follows:

\[([x], [y])\in R_1 \Leftrightarrow  (x, y)\cdot S \cdot (x, y)^t=\begin{pmatrix}
g&1\\
1&g^{-1}
\end{pmatrix},\]
and for $2\le i\le (q-1)/2$, 
\[([x], [y])\in R_i \Leftrightarrow  (x, y)\cdot S \cdot (x, y)^t=\begin{pmatrix}
g&1\\
1&g^{2i-3}
\end{pmatrix},\]
and 
\[([x], [y])\in R_{(q+1)/2} \Leftrightarrow  (x, y)\cdot S \cdot (x, y)^t=\begin{pmatrix}
g&0\\
0&g
\end{pmatrix},\]
where $g$ is a primitive element of $\mathbb{F}_q$, and $S$ is the associated matrix of $Q$. 

\paragraph{Example:} Suppose $d=5.$ Then we have 
\[Q(x)=2(x_1x_3+x_2x_4)+x_5^2.\]
The associated matrix $S$ of $Q$ is 
\[S=
\begin{pmatrix}
0&0&1&0&0\\
0&0&0&1&0\\
1&0&0&0&0\\
0&1&0&0&0\\
0&0&0&0&1
\end{pmatrix}
.\]
By a direct computation, we have 
\[\begin{pmatrix}
x  \\
y
\end{pmatrix}\cdot S \cdot \begin{pmatrix}
x \\
y
\end{pmatrix}^t=\begin{pmatrix}Q(x, x)&Q(x, y)\\Q(x, y)&Q(y, y)\end{pmatrix},\]
where $Q(x, y)=x\cdot S\cdot y^t$.

For each $1\le i\le (q+1)/2$, we will call $(V, R_i)$ the $i$-th association scheme graph. 
The spectrum of each graph $(V, R_i)$ has been studied in \cite{BSH} via the character tables of  the association scheme $\mathcal{X}(V, (R_i)_{0\le i\le (q+1)/2})$. In particular, the following is its character table. 

\[P=\begin{bmatrix}
    1 & (q^{m-1}-1)(q^m+1) & q^{m-1}(q^m+1) & \dots  & q^{m-1}(q^m+1)& \frac{1}{2}q^{m-1}(q^m+1) \\
    1 & -(q-2)q^{m-1}-1 & 2q^{m-1} & \dots  & 2q^{m-1}&q^{m-1} \\
     1 & q^{m-1}-1 &  &   &  \\
    \vdots & \vdots & & (q^{m-1}\chi_{ij})_{\substack{1\le i\le (q-1)/2\\ 1\le j\le (q-1)/2}} & \\
    1 & q^{m-1}-1 &  &   & 
\end{bmatrix},\]
where \[\chi_{ij}=\frac{1}{2}\left( \varphi_{i, 2j-1}+\varphi_{i, 2j}\right)\]
for $1\le i\le (q-1)/2$ and $1\le j\le (q-3)/2$, and 
\[\chi_{i, (q-1)/2}=\frac{1}{2}\varphi_{i, q-2}\]
for $1\le i\le (q-1)/2$, and 
\[|\varphi_{ij}|\le 2\sqrt{q}\]
for all $1\le i\le q-1$, $1\le j\le q-2$.

The character table of the association scheme $\mathcal{X}(V, (R_i)_{0\le i\le (q+1)/2})$ gives us a comprehensive interpretation of spectrum of the graphs $(V, R_i)$ with $1\le i\le (q+1)/2$. For example,  the graphs $(V, R_i)$, $2\le i\le (q-1)/2$, are regular graphs of degree $q^{2m-1}+q^{m-1}$, the graph $(V, R_{(q+1)/2})$ is a regular graph of degree $q^{m-1}(q^m+1)/2$, and the graph $(V, R_1)$ is a regular graph of degree $(q^{m-1}-1)\cdot (q^m+1)$. 

Notice that from the character table, we have the graphs $(V, R_i)$ are Ramanujan for all $i\ge 2$. If $i=1$, the first association scheme graph is Ramanujan only if $q=3, 5$ or $q=7$ and $m\ge 3$. We refer readers to \cite{banai} for the background of the theory of association schemes and applications in graph theory. 

In this paper, for our purpose, we only need to make use of the spectrum of the first association scheme graph $(V, R_1)$.

It follows from the character table that the distinct eigenvalues of $(V, R_1)$ are $$(q^{m-1}-1)\cdot (q^m+1),~ -(q-2)q^{m-1}-1,~ q^{m-1}-1.$$ 
In the following lemma, we will estimate the number of edges between in one set in the graph $(V, R_1)$.


%

\begin{lemma}\label{090}
Let $W$ be a vertex set in the graph $(V, R_1)$, and let $e(W, W)$ be the number of edges between $W$ and $W$. We have the following
\[e(W, W)\ll \frac{|W|^2}{q}+q^{m-1}|W|.\]
\end{lemma}
\begin{proof}
Suppose $|V|=n$. Let $A$ be the adjacency matrix of the graph $(V, R_1)$. Suppose that $v_1, \ldots, v_n$ are orthonormal eigenvectors of $A$ with $v_1=\frac{1}{\sqrt{n}}\cdot \textbf{1}$, where $\mathbf{1}=(1, \ldots, 1)$. Let $\lambda_i$ be eigenvalues corresponding to $v_i$, i.e $Av_i=\lambda_iv_i$. Since the graph $(V, R_1)$ is a regular graph of degree $(q^{m-1}-1)\cdot (q^m+1)$, we have $\lambda_1=(q^{m-1}-1)\cdot (q^m+1)$. 

Let $\mathbf{1}_W$ be the characteristic vector of $W$. Expanding $\mathbf{1}_W$ in the basis $\{v_i\}_i$. i. e. $\mathbf{1}_W=\sum_{i}\alpha_iv_i$. Therefore, 
\[e(W, W)=\mathbf{1}_W\cdot A\cdot \mathbf{1}_W=\left(\sum_{i}\alpha_iv_i\right)\cdot A\cdot \left(\sum_{i}\alpha_iv_i\right)=\sum_{i}\lambda_i\alpha_i^2.\]
On the other hand, we also have 
\[\alpha_1=<\mathbf{1}_W, v_1>=\frac{|W|}{\sqrt{n}}.\]
So, 
\[e(W, W)=\frac{\lambda_1}{n}|W|^2+\sum_{i\ge 2}\lambda_i\alpha_i^2\le \frac{\lambda_1}{n}|W|^2+\sum_{i\ge 2, \lambda_i>0}\lambda_i\alpha_i^2.\]
Since $n=q^m(q^m-1)/2$ and $\lambda_1=(q^{m-1}-1)\cdot (q^m+1)$, we get $\lambda_1/n\ll 1/q$. Moreover, the distinct eigenvalues of $(V, R_1)$ are $(q^{m-1}-1)\cdot (q^m+1), -(q-2)q^{m-1}-1, q^{m-1}-1$, which implies that 
\[e(W, W)\ll \frac{|W|^2}{q}+q^{m-1}|W|.\]
This completes the proof of the lemma. 
\end{proof}

As a consequence of Lemma \ref{090}, we have the following result which will be used directly in the proof of Theorem \ref{energy-sphere}.

\begin{lemma}\label{Adidaphat}
Suppose either $d=4k-1$ and $q\equiv 1\mod 4$ or $d=4k+1$. For any $A\subset S_g$, the number of pairs $(a, b)\in A^2$ such that $||a-b||=0$ is at most $|A|^2/q+q^{(d-3)/2}|A|$.
\end{lemma}
\begin{proof}
Since either $d=4k-1$ and $q\equiv 1\mod 4$ or $d=4k+1$, we have the form $Q(x)$ is equivalent with our distance function $x_1^2+\cdots+x_d^2$. Therefore, it is sufficient to prove this lemma with the form of $Q(x)$, i.e. we count the number of pairs $(a, b)\in A^2$ such that $Q(a, a)=g, Q(b, b)=g, Q(a, b)=g$. 

Define 
\[A':=\{y/g\colon y\in A\}.\]
For any element $a'\in A'$, we have $Q(a', a')=1/g$. Since $g$ is primitive element, it is not a square number. So $A'$ is a set of non-square non-isotropic elements. We might need to partition the set $A'$ into subsets such that there are no two elements on any line passing through the origin. It is clear that the number of such subsets is bounded by a constant. 

If $Q(a, a)=g, Q(b, b)=g, Q(a, b)=g$, then we have 
\[Q(a, a)=g, Q(b/g, b/g)=1/g, Q(a, b/g)=1.\]
In other words, we can say that the number of desired pairs $(a, b)$ is bounded by the number of edges between $A$ and $A'$ in the graph $(V, R_1)$. Since $|A|=|A'|$, we have $|A\cup A'|\le 2|A|$. We also have that the number of edges between $A$ and $A'$ is at most $e(A\cup A', A\cup A')$. By Lemma \ref{090}, that number is at most $\ll  |A|^2/q+q^{\frac{d-3}{2}}|A|.$ This concludes the proof of the lemma.
\end{proof}

We are now ready to prove Theorem \ref{energy-sphere}.
\paragraph{Proof of Theorem \ref{energy-sphere}.} We start with the following observation. Given $a, b, c, d\in S_g$, if $a+b=c+d$, then we have 
\[(b-d)\cdot (a-d)=0.\]
This can be viewed as a right angle at $d$. Thus $E(A)$ is bounded by the number of triples $(a, b, d)\in A^3$ such that $(b-d)\cdot (a-d)=0$. We now fall into two cases: 

{\bf Case $1$:} Let $E_1$ be the number of triples $(a, b, d)\in A^3$ such that either $||b-d||=0$ or $||a-d||=0$. We are going to show that 
\[E_1\ll \frac{|A|^3}{q}+q^{\frac{d-3}{2}}|A|^2.\]
It follows from Lemma \ref{Adidaphat} that the number of pairs $(a, b)\in A^2$ such that $||a-b||=0$ is at most $\ll |A|^2/q+q^{(d-3)/2}|A|$. Thus the number of such triples is at most 
\[\ll \frac{|A|^3}{q}+q^{(d-3)/2}|A|^2.\]
{\bf Case $2$:} Let $E_2$ be the number of triples $(a, b, d)\in A^3$ such that $||a-d||\ne 0, ||b-d||\ne 0$. 

For a fixed $d\in A$, we now count the number of pairs $(a, b)\in A^2$ such that $(a-d)\cdot (b-d)=0$. 

Since $||a-d||\ne 0$, there is no other point $a'\in A$ such that $a'-d=\lambda (a-d)$ for some $\lambda \in \mathbb{F}_q^*\setminus\{1\}$. The same also holds for $b-d$. 

Let $U$ be a point set in $\mathbb{F}_q^d$ defined by 
\[U:=\{\lambda\cdot (a-d)\colon a\in A, \lambda\ne 0\},\]
and $V$ be the set of hyperplanes defined by $(b-d)\cdot x=0$, where $b\in A$. Let $I(U, V)$ be the number of incidences between $U$ and $V$. 

One can check that if $(a-d)\cdot (b-d)=0$, then we have $(q-1)$ incidences between points in $\{\lambda\cdot (a-d)\colon \lambda\ne 0\}$ and the plane $(b-d)\cdot x=0$. Therefore, the number of pairs $(a, b)\in A^2$ such that $(a-d)\cdot (b-d)=0$ is at most $(q-1)^{-1}I(U, V)$. 

For any point set $U$ and any hyperplane set $V$ in $\mathbb{F}_q^d$, it has been proved in \cite{vinh} that 
\[I(U, V)\ll \frac{|U||V|}{q}+q^{\frac{d-1}{2}}\sqrt{|U||V|}.\]
Using the fact that $|U|=(q-1)|A|$ and $|V|=|A|$,  the number of pairs $(a, b)\in A^2$ such that $(a-d)\cdot (b-d)=0$ is at most
\[\frac{|A|^2}{q}+q^{\frac{d-2}{2}}|A|.\]
Summing over all $d\in A$, we have 
\[E_2\ll \frac{|A|^3}{q}+q^{\frac{d-2}{2}}|A|^2.\]
Putting $E_1$ and $E_2$ together, the theorem follows. $\square$



\section{Extension theorems for paraboloids (Theorem \ref{thm:main3})}\label{sec9}

Given a function $f\colon \mathbb{F}_q^d\to \mathbb{C},$ we denote its support by $S$. For any $z\in \mathbb{F}_q$, the function $S_z\colon P\to\{0, 1\}$ is defined by
\[S_z(\underline{x}, ||\underline{x}||)=1_{S}(\underline{x}, z).\] 
We will use the following lemmas in our proof. 

\begin{lemma}[Lemma $2.1$, \cite{hello}]\label{lm1-koh}
Let $f\colon \mathbb{F}_q^d\to \mathbb{C}$ such that $|f|\ll 1$ on its support. We have 
\[||\widetilde{f}||_{L^2(P, d\sigma)}\ll |S|^{1/2}+|S|^{3/8}q^{\frac{d-1}{4}}\left(\sum_{z\in \mathbb{F}_q}||(S_zd\sigma)^\vee ||_{L^4(\mathbb{F}_q^d, dc)}\right)^{1/2},\]
where $\widetilde{f}$ denotes the Fourier transform of $f$, which is defied by 
$$ \widetilde{f}(x)=\sum_{m\in \mathbb F_q^d} \chi(-m\cdot x) f(m).$$
\end{lemma}
\begin{lemma}\label{lm2}
For $A\subset P\subset \mathbb{F}_q^d$, we have 
\[||(Ad\sigma)^\vee||_{L^4(\mathbb{F}_q^d, dc)}=q^{\frac{4-3d}{4}}(E(A))^{1/4}.\]
\end{lemma}
\begin{proof}
It is enough to show that 
\[||(Ad\sigma)^\vee||_{L^4(\mathbb{F}_q^d, dc)}^4=q^{4-3d}E(A).\]
Indeed, 
\begin{align*}
||(Ad\sigma)^\vee||^4_{L^4(\mathbb{F}_q^d, dc)}&=\sum_{m}\left\vert(Ad\sigma)^\vee(m)\right\vert^4=\frac{1}{|P|^4}\sum_{m}\left\vert \sum_{x\in P}A(x)\chi(m\cdot x)\right\vert^4\\
&= \frac{1}{|P|^4}\sum_{m}\sum_{x, y, z, t\in P}A(x)A(y)A(z)A(t)\chi(m\cdot (x+y-z-t))\\
&=q^{4-3d}E(A).
\end{align*}
\end{proof}
Combining Lemma \ref{lm1-koh} and Theorem \ref{thm1}, we obtain the following.

\begin{lemma}\label{lm31-koh}
Let $f\colon \mathbb{F}_q^d\to \mathbb{C}$ such that $f\sim 1$ on its support $S\subset \mathbb{F}_q^d$. Then 
\[||\widetilde{f}||_{L^2(P, d\sigma)}\ll |S|^{1/2}+|S|^{3/4}q^{\frac{2-d}{8}}+q^{\frac{6-d}{16}}|S|^{5/8}.\]
\end{lemma}
\begin{proof}
It follows from Lemma \ref{lm1-koh} that 
\begin{align*}
||\widetilde{f}||_{L^2(P, d\sigma)}&\ll |S|^{1/2}+|S|^{3/8}q^{\frac{d-1}{4}}\left(\sum_{z\in \mathbb{F}_q}||(S_zd\sigma)^\vee ||_{L^4(\mathbb{F}_q^d, dc)}\right)^{1/2}\\
&\ll |S|^{1/2}+|S|^{3/8}q^{\frac{d-1}{4}}q^{\frac{4-3d}{8}}\left(q^{-1/4}\sum_{z\in \mathbb{F}_q}|S_z|^{3/4}+q^{\frac{d-2}{8}}\sum_{z\in \mathbb F_q}|S_z|^{1/2}\right)^{1/2}\\
&\ll |S|^{1/2}+|S|^{3/4}q^{\frac{2-d}{8}}+q^{\frac{6-d}{16}}|S|^{5/8},
\end{align*}
where in the last inequality, we have used Theorem \ref{thm1}, Lemma \ref{lm2}, and the Cauchy-Schwarz inequality.
\end{proof}

Using the finite field Stein-Tomas estimate and the Parseval inequality, we also have 
\[||\widetilde{f}||_{L^2(P, d\sigma)}\ll \begin{cases}q^{1/2}|S|^{1/2}\\ |S|^{1/2}+q^{\frac{1-d}{4}}|S|.\end{cases}\]

Combining these bounds and Lemma \ref{lm31-koh} implies the next lemma.
 
\begin{lemma}
Let $f\colon \mathbb{F}_q^d\to \mathbb{C}$ such that $f\sim 1$ on its support $S\subset \mathbb{F}_q^d$. Then 
\[||\widetilde{f}||_{L_2(P, d\sigma)}\ll \begin{cases} &|S|^{1/2}q^{1/2} \quad\mbox{if} ~~|S|\ge q^{\frac{d+2}{2}}\\
& |S|^{5/8}q^{\frac{6-d}{16}}\quad \mbox{if}~~q^{\frac{3d+2}{6}}<|S|<q^{\frac{d+2}{2}}\\
&|S|^{1/2}+q^{\frac{1-d}{4}}|S|\quad\mbox{if}~~|S|\le q^{\frac{3d+2}{6}}.
\end{cases}\]
\end{lemma}

We are now ready to prove Theorem \ref{thm:main3}. 
\paragraph{Proof of Theorem \ref{thm:main3}.}
By the duality, it is enough to show that $||\widetilde{f}||_{L^2(P, d\sigma)}\ll 1$ for all $f\colon \mathbb{F}_q^d\to \mathbb{C}$ with the assumption
\[\sum_{x\in \mathbb{F}_q^d}|f(x)|^{\frac{2d+4}{d+4}}=1.\]
Define $f_i=1_{\{x\colon f(x)\sim 1/2^i\}}$. We denote the support of $f_i$ by $A_i$. It follows from the above assumption that $|A_i|\le 2^{i\frac{2(d+4)}{d+4}}$.

To bound $||\widetilde{f}||_{L^2(P, d\sigma)}$ we do as follows: 
\begin{align*}
&||\widetilde{f}||_{L^2(P, d\sigma)}\ll \sum_{i=0}^{\log q}2^{-i}||\widetilde{f}_i||_{L^2(P, d\sigma)}\\
\ll& \sum_{\substack{0 \le i \le \log q\\ 2^{i\frac{2d+4}{d+4}}\le q^{\frac{3d+2}{6}}}}2^{-i}||\widetilde{f}_i ||_{L^2(P, d\sigma)}+\sum_{\substack{0 \le i \le \log q\\q^{\frac{3d+2}{6}}\le 2^{i\frac{2d+4}{d+4}}\le q^{\frac{d+2}{2}}}}2^{-i}||\widetilde{f}_i ||_{L^2(P, d\sigma)}
+\sum_{\substack{0 \le i \le \log q\\ 2^{i\frac{2d+4}{d+4}}\ge q^{\frac{d+2}{2}}}}2^{-i}||\widetilde{f}_i ||_{L^2(P, d\sigma)}\\
=&I+II+III.
\end{align*}
We now bound $I, II, III$. 

We have 
\begin{align*}I=&\sum_{\substack{0 \le i \le \log q\\2^{i\frac{2d+4}{d+4}}\le q^{\frac{3d+2}{6}}}}2^{-i}||\widetilde{f}_i ||_{L^2(P, d\sigma)}\ll \sum_{\substack{0 \le i \le \log q\\2^{i\frac{2d+4}{d+4}}\le q^{\frac{3d+2}{6}}}} \left(2^{-i}2^{i\frac{2d+4}{2d+8}}+q^{\frac{1-d}{4}}2^{-i}2^{i\frac{2d+4}{d+4}}\right)\\
&\ll 1+\sum_{\substack{0 \le i \le \log q\\2^{i\frac{2d+4}{d+4}}\le q^{\frac{3d+2}{6}}}}q^{\frac{6-d}{12d+24}}\ll 1,
\end{align*}
whenever $d\ge 7$. 
\begin{align*}
II=\sum_{\substack{0 \le i \le \log q\\q^{\frac{3d+2}{6}}\le 2^{i\frac{2d+4}{d+4}}\le q^{\frac{d+2}{2}}}}2^{-i}||\widetilde{f}_i ||_{L^2(P, d\sigma)}\ll \sum_{\substack{0 \le i \le \log q\\q^{\frac{3d+2}{6}}\le 2^{i\frac{2d+4}{d+4}}\le q^{\frac{d+2}{2}}}}q^{\frac{6-d}{16}}2^{i\frac{d-6}{4d+16}}\ll 1.
\end{align*}
\begin{align*}
III=\sum_{\substack{0 \le i \le \log q\\2^{i\frac{2d+4}{d+4}}\ge q^{\frac{d+2}{2}}}}2^{-i}||\widetilde{f}_i ||_{L^2(P, d\sigma)}\ll\sum_{\substack{0 \le i \le \log q\\2^{i\frac{2d+4}{d+4}}\ge q^{\frac{d+2}{2}}}}2^{-i}2^{i\frac{2d+4}{2d+8}} \ll q^{1/2}\sum_{\substack{0 \le i \le \log q\\2^{i\frac{2d+4}{d+4}}\ge q^{\frac{d+2}{2}}}}2^{i\left(\frac{-4}{2d+8}\right)}\ll 1.
\end{align*}
Putting these bounds together, the theorem follows. $\square$
\section{Connection with the Erd\H{o}s-Falconer distance problem (Theorem \ref{conn2})}\label{sec11}
\paragraph{Proof of Theorem \ref{conn2}:}
Let us recall the following formula from \cite{IR06}. For $A\subset \mathbb{F}_q^d$ with $|A|\gg q^{d/2}$, we have
\begin{equation}\label{congthuc1}|\Delta(A)|\gg \min \left\lbrace q, \frac{q}{M_A(q)}\right\rbrace,\end{equation}
where $M_A(q)$ is the finite field version of the Mattila integral given by 
\[M_A(q)=\frac{q^{3d+1}}{|A|^4}\sum_{j\in \mathbb{F}_q^*}\left(\sum_{m\in S_j }|\widehat{A}(m)|^2\right)^2\le \frac{q^{d}}{|A|^3}\max_{j\ne 0}\| \widetilde{A}\|^2_{L^2(S_j, d\sigma)}.\]
In order to deduce the distance result, we need to find a good upper bound of $\sum_{m\in S_t} |\widehat{A}(m)|^2$ for any $t\ne 0.$ We can write that
\begin{equation}\label{BFP} \sum_{m\in S_t} |\widehat{A}(m)|^2 =q^{-2d} \sum_{x\in S_t} |\widetilde{A}(x)|^2 =q^{-2d} \sum_{X:=(x,s)\in P \subset \mathbb F_q^{d+1}} |\widetilde{A}(x) \delta_t(s)|^2,\end{equation}
where $P$ denotes the paraboloid in $\mathbb F_q^{d+1},$ and $\delta_t(s)=1$ for $s=t$, and $0$ otherwise. Notice that $\delta_t(s)=\widetilde{(\delta_t)^\vee}(s),$ where $(\delta_t)^\vee(k)=q^{-1} \sum_{s\in \mathbb F_q} \chi(ks) \delta_t(s) = q^{-1} \chi(kt):=q^{-1} \chi_t(k).$ In addition, observe that
$$ \widetilde{A}(x) \delta_t(s) = \widetilde{A}(x) \widetilde{(\delta_t)^\vee}(s) = \widetilde{A\otimes (\delta_t)^\vee} (x, s),$$
where we define $(f\otimes g)(x,s):=f(x) g(s).$

From \eqref{BFP} and the above observations, we see that
\begin{align}\label{L2R}\nonumber\sum_{m\in S_t} |\widehat{A}(m)|^2 &= q^{-2d} q^d \frac{1}{|P|} \sum_{(m,s)\in P\subset \mathbb F_q^d\times \mathbb F_q} \left|\widetilde{A\otimes (\delta_t)^\vee}(m,s)\right|^2
\\
&=q^{-d} \|\widetilde{A\otimes (\delta_t)^\vee}\|^2_{L^2(P, d\sigma)}
=q^{-d-2} \|\widetilde{A\otimes \chi_t}\|^2_{L^2(P, d\sigma)}, \end{align}
where $d\sigma$ denotes the normalized surface measure on the paraboloid $P\subset \mathbb F_q^{d+1}.$
Hence we see that the distance problem in $\mathbb F_q^d$ is reduced to the $L^2$ restriction estimate for the paraboloid $P\subset \mathbb F_q^{d+1}$ in the specific case when the test function is $A\otimes \chi_t.$   In other words, we have proved that for $A\subset \mathbb{F}_q^d$ with $|A|\gg q^{d/2}$,  we have 
\[|\Delta(A)|\gg q\]
provided that  
$$\frac{q^{d-1}}{|A|^3} \max_{t\in \mathbb F_q^*} \|\widetilde{A\otimes \chi_t}\|^2_{L^2(P, d\sigma)}\ll 1,$$
which is a direct consequence of the paraboloid restriction conjecture. $\hfill\square$
\section{Extension theorems for spheres (Theorems \ref{thm:main3'} and \ref{thm:main3''})}\label{sec10}
We start this section with the following lemma.

\begin{lemma}\label{Adidaphat3}
Let $g$ be a primitive element of $\mathbb{F}_q$, and $S_g$ the sphere of radius $g$ centered at the origin. Let $d=4k+1$ or $d=4k-1$ and $q\equiv 1\mod 4$. For $A\subset S_g$ of size $n$, we have
$$ \|(Ad\sigma)^\vee\|_{L^{4}(\mathbb F_q^d, dc)} 
\ll \left\{  \begin{array}{ll} n^{\frac{3}{4}}q^{\frac{-3d+3}{4}}  \quad&\mbox{for}~~q^{\frac{d}{2}} \le n\ll q^{d-1}\\
n^{\frac{1}{2}}q^{\frac{-5d+6}{8}} \quad&\mbox{for}~~ q^{\frac{d-2}{2}} \le n\le q^{\frac{d}{2}}\\
n^{\frac{3}{4}}q^{\frac{-3d+4}{4}} \quad&\mbox{for}~~ 1 \le n\le q^{\frac{d-2}{2}},\end{array}\right.
$$
\end{lemma}
\begin{proof} Using the orthogonal property of $\chi$, we have
\[ \|(Ad\sigma)^\vee\|_{L^{4}(\mathbb F_q^d, dc)} =\frac{q^{\frac{d}{4}}}{|S_g|}\cdot  E(A)^{1/4} \sim q^{\frac{-3d+4}{4}}E(A)^{1/4}.\] 
We now fall into two cases:

{\bf Case $1$:}
 If $q^{\frac{d-2}{2}}\le n\ll q^{d-1}$, then we can apply Theorem \ref{energy-sphere} to get the desired bounds. 

{\bf Case $2$:} If $n\le q^{\frac{d-2}{2}}$, then we use the trivial bound $n^3$ for the energy to conclude the proof.
\end{proof}

\paragraph{Proof of Theorem \ref{thm:main3'} and Theorem \ref{thm:main3''} .}
By a direct computation, we can see that Theorems \ref{thm:main3'} and \ref{thm:main3''} are equivalent with 
\begin{equation}\label{sphere1} \|(fd\sigma)^\vee\|_{L^{4}(\mathbb F_q^d, dc)} \ll \|f\|_{L^{4d/(3d-2)}(S_g, d\sigma)}\sim \left(q^{-d+1} \sum_{x\in S_g} |f(x)|^{\frac{4d}{3d-2}}\right)^{\frac{3d-2}{4d}}.\end{equation}
This implies that 
\[q^{\frac{3d^2-5d+2}{4d}}\|(fd\sigma)^\vee\|_{L^{4}(\mathbb F_q^d, dc)}\ll \left(\sum_{x\in S_g} |f(x)|^{\frac{4d}{3d-2}}\right)^{\frac{3d-2}{4d}}.\]
Normalizing the function $f$ if necessary, we may assume that 
\begin{equation}\label{Adidaphat5}\sum_{x\in S_g}|f(x)|^{\frac{4d}{3d-2}}=1.\end{equation}
Therefore, it is sufficient to show that 
\[T:=q^{\frac{3d^2-5d+2}{4d}} \|(fd\sigma)^\vee\|_{L^{4}(\mathbb F_q^d, dc)} \ll 1.\]
We now decompose the function $f$ as follows:
\begin{equation}\label{Adidaphat4} f(x)=\sum_{i=0}^{\log q} 2^{-i} {A_i(x)},\end{equation}
where  $\{A_i\}$ are disjoint subsets of $S_g.$

It follows from \eqref{Adidaphat5} and \eqref{Adidaphat4} that  
$$ \sum_{i=0}^{\log q} 2^{-\frac{4d}{3d-2}i} |A_i| =1,$$ 
which gives us
\begin{equation}\label{Adidaphat6} |A_i|\le  2^{\frac{4d}{3d-2}i},\quad~~ \forall i.\end{equation}
We now bound $T$ as follows:
\begin{align*}T=&q^{\frac{3d^2-5d+2}{4d}} \|(fd\sigma)^\vee\|_{L^{4}(\mathbb F_q^d, dc)} \le q^{\frac{3d^2-5d+2}{4d}} \sum_{i=0}^{\log q} 2^{-i} \|({A_i}d\sigma)^\vee\|_{L^{4}(\mathbb F_q^d, dc)}\\
\ll&  q^{\frac{3d^2-5d+2}{4d}}\sum_{\substack{0 \le i \le \log q\\ 1\le 2^{\frac{4d}{3d-2}i} \le  q^{\frac{d-2}{2}}}} 2^{-i} \|({A_i}d\sigma)^\vee\|_{L^{4}(\mathbb F_q^d, dc)} 
 +  q^{\frac{3d^2-5d+2}{4d}}\sum_{\substack{0 \le i \le \log q\\ q^{\frac{d-2}{2}} \le 2^{\frac{4d}{3d-2}i} \le  q^{\frac{d}{2}}}} 2^{-i} \|({A_i}d\sigma)^\vee\|_{L^{4}(\mathbb F_q^d, dc)}\\
 &+q^{\frac{3d^2-5d+2}{4d}}\sum_{\substack{0 \le i \le \log q\\ q^{\frac{d}{2}}\le 2^{\frac{4d}{3d-2}i} \ll  q^{d-1}}} 2^{-i} \|({A_i}d\sigma)^\vee\|_{L^{4}(\mathbb F_q^d, dc)}\\
 &=: T_1 + T_1+T_3.
 \end{align*} 

Employing Lemma \ref{Adidaphat3}, we get
\[ T_1 \ll q^{\frac{-d+2}{4d}} \sum_{\substack{0 \le i \le \log q\\ 1\le 2^{\frac{4d}{3d-2}i} \le  q^{\frac{d-2}{2}}}} 2^{-i} |A_i|^{\frac{3}{4}} 
\ll q^{\frac{-d+2}{4d}} \sum_{\substack{0 \le i \le \log q\\ 1\le 2^{\frac{4d}{3d-2}i} \le q^{\frac{d-2}{2}}}} 2^{\frac{2}{3d-2}i} \ll q^{\frac{-d+2}{4d}} \cdot q^{\frac{d-2}{4d}}=1,\]

\[ T_2 \ll q^{\frac{d^2-4d+4}{8d}}  \sum_{\substack{0 \le i \le \log q\\ q^{\frac{d-2}{2}} \le 2^{\frac{4d}{3d-2}i} \le  q^{\frac{d}{2}}}} 2^{-i} |A_i|^{\frac{1}{2}} 
\ll q^{\frac{d^2-4d+4}{8d}}  \sum_{\substack{0 \le i\le \log q\\ q^{\frac{d-2}{2}} \le 2^{\frac{4d}{3d-2}i} \le  q^{\frac{d}{2}}}} 2^{\frac{-d+2}{3d-2}i} \ll q^{\frac{d^2-4d+4}{8d}} \cdot q^{\frac{-d^2+4d-4}{8d}} =1,\]
and 
\[T_3 \ll q^{\frac{-d+1}{2d}}\sum_{\substack{0 \le i \le \log q\\ q^{\frac{d}{2}}\le 2^{\frac{4d}{3d-2}i} \ll  q^{d-1}}} 2^{-i} |A_i|^{\frac{3}{4}} 
\ll  q^{\frac{-d+1}{2d}}\sum_{\substack{0 \le i \le \log q\\ q^{\frac{d}{2}}\le 2^{\frac{4d}{3d-2}i} \ll  q^{d-1}}} 2^{\frac{2}{3d-2}i} \ll  q^{\frac{-d+1}{2d}}\cdot q^{\frac{d-1}{2d}} =1. \]
This completes the proof of the theorems. $\square$

\section{Distances between a set on a variety and an arbitrary set}\label{sec7}
We recall that for $A,B\subset \mathbb F_q^d,$ the distance set between $A$ to $B$, denoted by $\Delta(A,B),$ is defined by
$$ \Delta(A,B):=\{\|a-b\|: a\in A, b\in B\}.$$

In this section we prove our main theorems on distance problems. We begin by proving  a preliminary key lemma.
For $t\in \mathbb F_q,$ we denote by $\mu(t)$ the number of pairs $(a,b)\in A\times B$ such that $\|a-b\|=t.$
Since $|A||B| =\sum_{t\in \Delta(A,B)} \mu(t)$, the Cauchy-Schwarz inequality yields that
\begin{equation}\label{BForm} |\Delta(A,B)|\ge \frac{|A|^2|B|^2}{\sum_{t\in \mathbb F_q} \mu^2(t)}.\end{equation}
Hence the distance problem between $A$ and $B$ can be reduced to the estimate of $\sum_{t\in \mathbb F_q} \mu^2(t).$\\

The following lemma plays a crucial role in deducing distance results on two sets, one of which lies on an algebraic variety.

\begin{lemma} \label{AveF} Let $A, B\subset \mathbb F_q^d.$ Then we have
\begin{equation}\label{AveFeq1} \sum_{t\in \mathbb F_q} \mu^2(t) \le \frac{|A|^2|B|^2}{q} + \frac{|A|}{q} \sum_{a\in A, s\ne 0} \left| \sum_{b\in B} \chi(2sa\cdot b - s\|b\|)\right|^2.\end{equation}
Moreover, for any set $\Omega$ with $\Omega\supset A,$  we have
\begin{equation}\label{AveFeq2}\sum_{t\in \mathbb F_q} \mu^2(t) \le \frac{|A|^2|B|^2}{q} + q^{d-1}|A| \sum_{b,b'\in B, s\ne 0} \widehat{\Omega}(2s(b'-b)) ~\chi(s(\|b'\|-\|b\|)).\end{equation}
\end{lemma}
\begin{proof} We first claim that \eqref{AveFeq2} follows from \eqref{AveFeq1}. To see this, assume that \eqref{AveFeq1} holds. If $A\subset \Omega$, then we can dominate the sum over $a\in A$ in \eqref{AveFeq1} by the sum over $a\in \Omega.$ Next, by expanding the square term and using the definition of the Fourier transform on $\Omega$, we obtain the inequality \eqref{AveFeq2} as claimed. Thus it only remains to prove \eqref{AveFeq1}.

We write 
$\mu(t)=\sum_{a\in A} 1 \times \left( \sum_{b\in B:\|a-b\|=t} 1\right).$
By the Cauchy-Schwarz inequality,
$$ \mu^2(t)\le |A| \sum_{a\in A} \left( \sum_{b\in B:\|a-b\|=t} 1\right)^2=|A|\sum_{a\in A, b,b'\in B: \|a-b\|=t=\|a-b'|} 1.$$
Summing over $t\in \mathbb F_q,$ it follows that
$$ \sum_{t\in \mathbb F_q} \mu^2(t) \le |A| \sum_{a\in A, b,b'\in B: \|a-b\|=\|a-b'\|} 1.$$
Since the condition $\|a-b\|=\|a-b'\|$ is equivalent to the condition $2a\cdot (b-b')=\|b\|-\|b'\|,$  we have
\begin{equation}\label{Formbasic} \sum_{t\in \mathbb F_q} \mu^2(t) \le |A| \sum_{a\in A, b,b'\in B: 2a\cdot (b-b')=\|b\|-\|b'\|} 1.\end{equation}

By the orthogonality of $\chi$, we can write
$$ \sum_{t\in \mathbb F_q} \mu^2(t) \le |A| q^{-1} \sum_{a\in A, b,b'\in B} \sum_{s\in \mathbb F_q} \chi\left(s(2a\cdot (b-b')-\|b\|+\|b'\|)\right).$$
Decompose the sum over $s\in \mathbb F_q$ into two parts: $s=0$ and $s\ne 0$. It follows that
$$ \sum_{t\in \mathbb F_q} \mu^2(t) \le |A|^2|B|^2 q^{-1} + |A| q^{-1} \sum_{a\in A, b,b'\in B} \sum_{s\ne 0} \chi\left(s(2a\cdot (b-b')-\|b\|+\|b'\|)\right).$$
Notice that the second term in the right-hand side of the above inequation can be written as
$$ |A|q^{-1} \sum_{a\in A, s\ne 0} \left|\sum_{b\in B} \chi(-s\|b\|+2sa\cdot b)\right|^2.$$ 
This completes the proof of \eqref{AveFeq1} and thus, the statement of the lemma follows. 
\end{proof}


\subsection{Proof of distance results on paraboloids (Theorem \ref{mainRP})}

We begin by reviewing the explicit form of the Fourier transform on the paraboloid (see \cite{IK09}).

\begin{lemma}\label{explicit}
Let $P \subset {\mathbb F}_q^d$ be the paraboloid. Then for each 
$m=(\underline{m}, m_d) \in {\mathbb F}_q^{d-1}\times {\mathbb F}_q,$  we have
$$ \widehat{P}(m)= \left\{\begin{array}{ll} q^{-d} \chi \left( \frac{\|\underline{m}\| }{4m_d}\right)
\eta^{d-1}(m_d) G_1^{d-1} 
 \quad &\mbox{if} \quad m_d \ne 0\\
0 \quad &\mbox{if} \quad m_d =0 , \underline{m} \ne \underline{0}\\
q^{-1} \quad &\mbox{if} \quad m=(0,\ldots,0),\end{array}\right.$$
where $\eta$ denotes the quadratic character of $\mathbb F_q$, and $G_1$ is the standard Gauss sum.                                          
\end{lemma}
We shall estimate the upper bound of the quantity $\sum_{t\in \mathbb F_q} \mu^2(t).$ To this end, from the formula \eqref{AveFeq2} of Lemma \ref{AveF} taking $\Omega$ as the paraboloid $P$ and then applying the explicit form of the Fourier transform on $P$ in Lemma \ref{explicit},
 we see that
\begin{align*}\sum_{t\in \mathbb F_q} \mu^2(t) \le& |A|^2|B|^2 q^{-1} + |A| q^{d-1} \sum_{s\ne 0, b\in B} \widehat{P}(0,\ldots,0)\\
 &+ |A| q^{-1} G_1^{d-1} \sum_{\substack{s\ne 0, b,b'\in B:\\ b'_d\ne b_d}} \chi(s(\|b'\|-\|b\|)) \chi\left( \frac{\|2s(\underline{b'}-\underline{b})\|}{8s(b'_d-b_d)}\right) \eta^{d-1}(2s(b'_d-b_d)).\end{align*}
 Since $\widehat{P}(0,\ldots, 0)=q^{-1}$, it follows that
 \begin{align}\label{KohE} \sum_{t\in \mathbb F_q} \mu^2(t)\le& |A|^2|B|^2 q^{-1} + |A||B| q^{d-1} \\ \nonumber
 &+ |A|q^{-1} G_1^{d-1}  \sum_{\substack{b, b'\in B:\\  b'_d\ne b_d}} \eta^{d-1}(2(b'_d-b_d)) \sum_{s\ne 0} \eta^{d-1}(s) \chi\left(s\left( \|b'\|-\|b\|+\frac{\|\underline{b'}-\underline{b}\|}{2(b'_d-b_d)}\right)\right).\end{align}

\textbf{Case 1:} Assume that $q\equiv 3~~(\mbox{mod}~4)$ and $d=4k-1$ for $k\in \mathbb N.$ We invoke the explicit value of the standard Gauss sum $G_1:= \sum_{s\ne 0} \eta(s) \chi(s).$ 
 
\begin{lemma}[\cite{LN97}, Theorem 5.15]\label{ExplicitGauss}
Let $\mathbb F_q$ be a finite field with $ q= p^{\ell},$ where $p$ is an odd prime and $\ell \in {\mathbb N}.$
Then we have
$$G_1= \left\{\begin{array}{ll}  {(-1)}^{\ell-1} q^{\frac{1}{2}} \quad &\mbox{if} \quad p \equiv 1 \mod 4 \\
{(-1)}^{\ell-1} i^\ell q^{\frac{1}{2}} \quad &\mbox{if} \quad p\equiv 3 \mod 4.\end{array}\right. $$
\end{lemma}
Since $q=p^\ell\equiv 3~~(\mbox{mod}~4)$ and $d=4k-1$ for $k\in \mathbb N$, we easily see from Lemma \ref{ExplicitGauss} that
\begin{equation*}\label{Gasssign}
G_1^{d-1}=-q^{\frac{d-1}{2}}. \end{equation*}

Since $d$ is odd,  $\eta^{d-1}\equiv 1.$ It therefore follows from \eqref{KohE} that
\begin{align*}\sum_{t\in \mathbb F_q} \mu^2(t) 
 \le& |A|^2|B|^2 q^{-1} + |A||B| q^{d-1}\\&  - |A|q^{-1} q^{\frac{d-1}{2}} \sum_{\substack{b, b'\in B:\\  b'_d\ne b_d}} \sum_{s\ne 0}  \chi\left(s\left( \|b'\|-\|b\|+\frac{\|\underline{b'}-\underline{b}\|}{2(b'_d-b_d)}\right)\right).\end{align*}
 Now observe that the sum over $s\ne 0$ is $(q-1)$ in the case when
 $$\|b'\|-\|b\|+\frac{\|\underline{b'}-\underline{b}\|}{2(b'_d-b_d)}=0.$$
 In this case, the contribution to the third term above is negative.
 On the other hand,  the sum over $s\ne 0$ is -1  if 
 $$\|b'\|-\|b\|+\frac{\|\underline{b'}-\underline{b}\|}{2(b'_d-b_d)}\ne 0.$$
In this case, the contribution of the third term above is positive. Thus, we see that 
 \begin{align*} \sum_{t\in \mathbb F_q} \mu^2(t) 
 \le  |A|^2|B|^2 q^{-1} + |A||B| q^{d-1} + |A||B|^2 q^{\frac{d-3}{2}}.\end{align*}
Combining this estimate with \eqref{BForm}, we obtain that
 $$|\Delta(A, B)|\ge \frac{1}{3} \min\left\{ q, \frac{|A||B|}{q^{d-1}}, \frac{|A|}{q^{\frac{d-3}{2}}} \right\}.$$
This implies that  if $|A|\ge q^{(d-1)/2}$ and $|A||B|\ge q^d$,  then $|\Delta(A, B)|\ge q/3.$

If $|A|\le q^{(d-1)/2}$, then Theorem $3.3$ in \cite{sun} states that $|\Delta(A, B)|\ge q/2$, whenever $|A||B|\ge 4q^d$. Thus, the proof of the first part of Theorem \ref{mainRP} is complete.\\

\textbf{Case 2:} Assume that $d\ge 4$ is even. Then $\eta^{d-1}=\eta.$
Thus, it follows from \eqref{KohE} that
\begin{align*} \sum_{t\in \mathbb F_q} \mu^2(t) \le& |A|^2|B|^2 q^{-1} + |A||B| q^{d-1} \\
 &+ |A|q^{-1} G_1^{d-1} \sum_{\substack{b, b'\in B:\\  b'_d\ne b_d}} \eta(2(b'_d-b_d)) \sum_{s\ne 0} \eta(s) \chi\left(s\left( \|b'\|-\|b\|+\frac{\|\underline{b'}-\underline{b}\|}{2(b'_d-b_d)}\right)\right).\end{align*}
  
Since the sum over $s\ne 0$ is a Gauss sum whose absolute value is less than or equal to $\sqrt{q}$, and $|G_1|=\sqrt{q}$,  we obtain that
$$ \sum_{t\in \mathbb F_q} \mu^2(t) \le |A|^2|B|^2 q^{-1} + |A||B| q^{d-1} + |A||B|^2 q^{\frac{d-2}{2}}.$$
Combining this estimate with \eqref{BForm}, we have
$$|\Delta(A,B)|\ge \frac{1}{3} \min\left\{ q, \frac{|A||B|}{q^{d-1}}, \frac{|A|}{q^{\frac{d-2}{2}}} \right\},$$
which implies that if $|A|\ge q^{d/2}$ and $|A||B|\ge q^d,$ then $|\Delta(E)|\ge q/3.$

If $|A|\le q^{(d-1)/2}$, then Theorem $3.5$ in \cite{sun} states that $|\Delta(A, B)|\ge q/144$, whenever $|A||B|\ge 16q^d$. Thus the proof of the second part of Theorem \ref{mainRP} is complete.

\subsection{Proof of distance results on spheres (Theorem \ref{mainRSO}, \ref{mainRSE}, \ref{mainRZS})}
We begin by proving the following lemma.

\begin{lemma}\label{SForm} Let $A$ be a set of the sphere $S_j$ in $\mathbb F_q^d, $ and $B$ be a set of $\mathbb F_q^d.$ For each $t\in \mathbb F_q,$ we denote by $\mu(t)$ the number of the pairs $(a,b)\in A\times B$ such that $\|a-b\|=t.$ Then, we have
\begin{align*} \sum_{t\in \mathbb F_q} \mu^2(t) \le& \frac{|A|^2|B|^2}{q} + q^{d-1}|A||B| \\
&+ q^{-2}\eta^d(-1)G_1^d |A| \sum_{b,b'\in B, s,r\ne 0} \eta^d(r) \chi\left(jr+ \frac{s^2\|b'-b\|}{r}\right) \chi(s(\|b'\|-\|b\|),\end{align*}
where $\eta$ is the quadratic character of $\mathbb F^*_q$, and $G_1:=\sum_{t\in \mathbb F_q^*}\eta(t)\chi(t)$ is the standard Gauss sum.
\end{lemma}
\begin{proof}
We need the following fact which was proved in \cite{I-K}. We have
$$ \widehat{S_j}(m)=\frac{\delta_0(m)}{q} + q^{-d-1}\eta^d(-1) G_1^d \sum_{r\ne 0} \eta^d(r) \chi\left(jr+\frac{\|m\|}{4r}\right),$$
for $m\in \mathbb F_q^d$ and $j\in \mathbb F_q.$
 
Inserting this formula for $\widehat{S_j}$ into \eqref{AveFeq2} of Lemma \ref{AveF} with $\Omega=S_j,$ we see from  a direct computation that 
\begin{align*}\sum_{t\in \mathbb F_q} \mu^2(t) =&\frac{|A|^2|B|^2}{q} + q^{d-2}|A| \sum_{b,b'\in B, s\ne 0} \delta_0(2s(b'-b)) ~\chi(s(\|b\|-\|b'\|))\\
&+ q^{-2}\eta^d(-1) G_1^d |A| \sum_{b,b'\in B, s,r\ne 0} \eta^d(r) \chi\left(jr+\frac{s^2\|b'-b\|}{r}\right)~\chi(s(\|b'\|-\|b\|)). \end{align*} 
It is not hard to see that  the second term of the right hand side above is dominated by $ q^{d-1}|A||B|.$
Therefore, the statement of the lemma follows.
\end{proof}

\paragraph{Proof of Theorem \ref{mainRSO}:}
As in the proof of Theorem \ref{mainRP}, we know from \cite[Theorems 3.3, 3.5]{sun} that 

\begin{enumerate}
\item If $d\ge 3$ is odd and $|A|\le q^{(d-1)/2}$, then $|\Delta(A, B)|\ge q/2$ under the condition $|A||B|\ge 4q^d$.
\item If $d\ge 4$ is even and $|A|\le q^{(d-1)/2}$, then $|\Delta(A, B)|\ge q/144$ under the condition $|A||B|\ge 16q^d$.
\end{enumerate}

Therefore, in the rest of proof, we assume that $|A|\ge q^{(d-1)/2}$ in odd dimensions and $|A|\ge q^{d/2}$ in even dimensions.

Since $\eta^d= \eta$ for odd $d$, Lemma \ref{SForm} gives 
\begin{equation} \label{Thangnonzero}\sum_{t\in \mathbb F_q} \mu^2(t) \le \frac{|A|^2|B|^2}{q} + q^{d-1}|A||B| + M,\end{equation}
where $M$ is given by
$$M:= q^{-2}\eta(-1) G_1^d |A| \sum_{b,b'\in B, s,r\ne 0} \eta(r) \chi\left(jr+ \frac{s^2\|b'-b\|}{r}\right) \chi(s(\|b'\|-\|b\|).$$ 
To bound $M$, we write $M$ as follows: 
\begin{align*} M&:=M_1+ M_2+M_3+M_4\\
&:= q^{-2}\eta(-1) G_1^d |A| \sum_{\substack{b,b'\in B, s,r\ne 0:\\ \|b'-b\|=0, \|b\|=\|b'\|}} \Omega(b,b',s,r) + q^{-2}\eta(-1) G_1^d |A| \sum_{\substack{b,b'\in B, s,r\ne 0:\\ \|b'-b\|=0, \|b\|\ne\|b'\|}} \Omega(b,b',s,r) \\
&+ q^{-2}\eta(-1) G_1^d |A| \sum_{\substack{b,b'\in B, s,r\ne 0:\\ \|b'-b\|\ne 0, \|b\|=\|b'\|}} \Omega(b,b',s,r) +  q^{-2}\eta(-1) G_1^d |A|\sum_{\substack{b,b'\in B, s,r\ne 0:\\ \|b'-b\|\ne 0, \|b\|\ne\|b'\|}}  \Omega(b,b',s,r),\end{align*}
where $\Omega(b,b',s,r):=\eta(r) \chi\left(jr+ \frac{s^2\|b'-b\|}{r}\right) \chi(s(\|b'\|-\|b\|)$ for $b,b'\in B, s,r\in \mathbb F_q^*.$
We will apply the following basic Gauss sum estimates: for $u, v, w\ne 0,$
\begin{equation} \label{formb} \sum_{s\in \mathbb F_q} \chi(u s^2)=\eta(u) G_1 \quad \mbox{and}\quad \sum_{r\ne 0} \eta(vr) \chi(w r) =\sum_{r\ne 0} \eta(vr^{-1}) \chi(w r)= \eta(vw ) G_1,\end{equation}
where we recall that the quadratic character $\eta$ satisfies that 
$\eta(r)=\eta(r^{-1})$ for $r\ne 0.$
It follows that
\begin{align*} M_1&= q^{-2}\eta(-1) G_1^d |A| \sum_{\substack{b,b'\in B, s,r\ne 0:\\ \|b'-b\|=0, \|b\|=\|b'\|}}\eta(r) \chi(jr)\\
&=q^{-2}\eta(-j) G_1^{d+1} |A| (q-1) \sum_{\substack{b,b'\in B:\\ \|b'-b\|=0, \|b\|=\|b'\|}}1,\end{align*}
where we applied the second formula in \eqref{formb}. Hence, we can write

\begin{equation} \label{eqM1} M_1=q^{-1} \eta(-j) G_1^{d+1} |A| \sum_{\substack{b,b'\in B:\\ \|b'-b\|=0, \|b\|=\|b'\|}}1 -q^{-2} \eta(-j) G_1^{d+1} |A| \sum_{\substack{b,b'\in B:\\ \|b'-b\|=0, \|b\|=\|b'\|}}1. \end{equation}
Next, notice that $M_2$ can be written as
$$ M_2=q^{-2}\eta(-1) G_1^d |A| \sum_{\substack{b,b'\in B:\\ \|b'-b\|=0, \|b\|\ne\|b'\|}} \left(\sum_{r\ne 0} \eta(r)\chi(jr) \right) \left(\sum_{s\ne 0} \chi(s(\|b'\|-\|b\|))\right).$$
Since the sum over $r\ne 0$ is $\eta(j)G_1$ and the sum over $s\ne 0$ is $-1$, we obtain that
\begin{equation}\label{eqM2} M_2=-q^{-2}\eta(-j) G_1^{d+1} |A| \sum_{\substack{b,b'\in B:\\ \|b'-b\|=0, \|b\|\ne\|b'\|}} 1. \end{equation}

Now, observe that $M_3$ can be written as
$$ M_3=q^{-2}\eta(-1) G_1^d |A| \sum_{\substack{b,b'\in B, s,r\ne 0:\\ \|b'-b\|\ne 0, \|b\|=\|b'\|}}\eta(r) \chi\left(jr+ \frac{s^2\|b'-b\|}{r}\right)$$
$$=q^{-2}\eta(-1) G_1^d |A|\sum_{\substack{b,b'\in B, r\ne 0:\\ \|b'-b\|\ne 0, \|b\|=\|b'\|}} \eta(r) \chi(jr)\sum_{s\in \mathbb F_q}  \chi\left(\frac{s^2\|b'-b\|}{r}\right) $$
$$-q^{-2}\eta(-1) G_1^d |A| \sum_{\substack{b,b'\in B, r\ne 0:\\ \|b'-b\|\ne 0, \|b\|=\|b'\|}}\eta(r) \chi(jr).$$ 

Compute the first term above by applying the first equation in \eqref{formb} and using the fact that $\sum_{r\ne 0} \chi(jr)=-1$ for $j\ne 0.$ Then, we see that 
\begin{equation}\label{eqM3}
M_3=-q^2\eta(-1)G_1^{d+1}|A| \sum_{\substack{b,b'\in B:\\ \|b'-b\|\ne 0, \|b\|=\|b'\|}} \eta(\|b'-b\|) -q^{-2}\eta(-j) G_1^{d+1} |A| \sum_{\substack{b,b'\in B:\\ \|b'-b\|\ne 0, \|b\|=\|b'\|}} 1.
\end{equation} 
Finally, we estimate $M_4$ which is given by
\begin{align*} M_4=&q^{-2}\eta(-1) G_1^d |A|\sum_{\substack{b,b'\in B, s,r\ne 0:\\ \|b'-b\|\ne 0, \|b\|\ne\|b'\|}}
\eta(r) \chi\left(jr+ \frac{s^2\|b'-b\|}{r}\right) \chi(s(\|b'\|-\|b\|))\\
=& q^{-2}\eta(-1) G_1^d |A|\sum_{\substack{b,b'\in B, r\ne 0:\\ \|b'-b\|\ne 0, \|b\|\ne\|b'\|}} \eta(r)\chi(jr) \sum_{s\in  \mathbb F_q} \chi\left( \frac{s^2\|b'-b\|}{r}+s(\|b'\|-\|b\|) \right) \\
&-q^{-2}\eta(-1) G_1^d |A|\sum_{\substack{b,b'\in B, r\ne 0:\\ \|b'-b\|\ne 0, \|b\|\ne\|b'\|}} \eta(r)\chi(jr) .\end{align*}
Notice that the sum over $r\ne 0$ in the second term above is $\eta(j)G_1.$ In order to estimate the first term above,  compute the sum over $s\in \mathbb F_q$ by applying the following formula for the Gauss sum estimate (see \cite{IK09}, Lemma 5)
\begin{equation}\label{GRF} \sum_{s\in \mathbb F_q} \chi(u s^2 + v s) = \eta(u) G_1 ~\chi\left(\frac{v^2}{-4u}\right),\end{equation}
for $u\ne 0, v\in \mathbb F_q.$
 
Then, we see that the sum over $s\in \mathbb F_q$ in the first term above is the same as
$$ \eta\left(\frac{\|b'-b\|}{r}\right) G_1 \chi\left( \frac{(\|b'\|-\|b\|)^2r}{-4\|b'-b\|}\right).$$
It follows that 
\begin{align*} M_4=& q^{-2}\eta(-1) G_1^{d+1} |A|\sum_{\substack{b,b'\in B:\\ \|b'-b\|\ne 0, \|b\|\ne\|b'\|}} 
\eta(\|b'-b\|) \sum_{r\ne 0} \chi\left(\left(j-\frac{(\|b'\|-\|b\|)^2}{4\|b'-b\|}\right)r\right)\\
&-q^{-2}\eta(-j) G_1^{d+1} |A|\sum_{\substack{b,b'\in B:\\ \|b'-b\|\ne 0, \|b\|\ne\|b'\|}} 1. \end{align*}
Compute the sum over $r\ne 0$ above by considering two cases: $4j\|b'-b\|-(\|b'\|-\|b\|)^2=0$ and $4j\|b'-b\|-(\|b'\|-\|b\|)^2\ne 0$. 
We see that
$$M_4=q^{-2}\eta(-1) G_1^{d+1} |A| (q-1) \sum_{\substack{b,b'\in B:\\ \|b'-b\|\ne 0, \|b\|\ne\|b'\|,\\4j\|b'-b\|=(\|b'\|-\|b\|)^2}} \eta(\|b'-b\|)$$
$$-q^{-2}\eta(-1) G_1^{d+1} |A|\sum_{\substack{b,b'\in B:\\ \|b'-b\|\ne 0, \|b\|\ne\|b'\|,\\4j\|b'-b\|\ne (\|b'\|-\|b\|)^2}} \eta(\|b'-b\|)$$
$$-q^{-2}\eta(-j) G_1^{d+1} |A|\sum_{\substack{b,b'\in B:\\ \|b'-b\|\ne 0, \|b\|\ne\|b'\|}} 1.$$

By rearranging the first two terms above, we have
$$ M_4=q^{-1}\eta(-1) G_1^{d+1} |A| \sum_{\substack{b,b'\in B:\\ \|b'-b\|\ne 0, \|b\|\ne\|b'\|,\\4j\|b'-b\|=(\|b'\|-\|b\|)^2}} \eta(\|b'-b\|)$$
$$-q^{-2}\eta(-1) G_1^{d+1} |A|\sum_{\substack{b,b'\in B:\\ \|b'-b\|\ne 0, \|b\|\ne\|b'\|}} \eta(\|b'-b\|)
-q^{-2}\eta(-j) G_1^{d+1} |A|\sum_{\substack{b,b'\in B:\\ \|b'-b\|\ne 0, \|b\|\ne\|b'\|}} 1.$$
Now, in the first term above, we observe from the condition $4j\|b'-b\|\ne (\|b'\|-\|b\|)^2$ that $\eta(\|b'-b\|)=\eta(j).$ It follows from this observation that 
\begin{align*} M_4=&q^{-1}\eta(-j) G_1^{d+1} |A| \sum_{\substack{b,b'\in B:\\ \|b'-b\|\ne 0, \|b\|\ne\|b'\|,\\4j\|b'-b\|=(\|b'\|-\|b\|)^2}} 1
-q^{-2}\eta(-1) G_1^{d+1} |A|\sum_{\substack{b,b'\in B:\\ \|b'-b\|\ne 0, \|b\|\ne\|b'\|}} \eta(\|b'-b\|)\\
&-q^{-2}\eta(-j) G_1^{d+1} |A|\sum_{\substack{b,b'\in B:\\ \|b'-b\|\ne 0, \|b\|\ne\|b'\|}} 1.
\end{align*}
We sum this term $M_4$ and terms $M_1, M_2, M_3$ of \eqref{eqM1}, \eqref{eqM2}, \eqref{eqM3}, respectively. It follows that
\begin{align*} M=&M_1+M_2+M_3+M_4\\
=& q^{-1} \eta(-j) G_1^{d+1} |A| \left( \sum_{\substack{b,b'\in B:\\ \|b'-b\|=0, \|b\|=\|b'\|}}1 + \sum_{\substack{b,b'\in B:\\ \|b'-b\|\ne 0, \|b\|\ne\|b'\|,\\4j\|b'-b\|=(\|b'\|-\|b\|)^2}} 1\right) \\
&-q^{-2}\eta(-j)G_1^{d+1} |A| \left(\sum_{\substack{b,b'\in B:\\ \|b'-b\|=0, \|b\|=\|b'\|}} 1 
+ \sum_{\substack{b,b'\in B:\\ \|b'-b\|=0, \|b\|\ne\|b'\|}} 1
+ \sum_{\substack{b,b'\in B,:\\ \|b'-b\|\ne 0, \|b\|=\|b'\|}} 1
+ \sum_{\substack{b,b'\in B:\\ \|b'-b\|\ne 0, \|b\|\ne\|b'\|}}1\right)\\
&-q^{-2}\eta{(-1)}G_1^{d+1} |A| \left( \sum_{\substack{b,b'\in B,:\\ \|b'-b\|\ne 0, \|b\|=\|b'\|}} \eta(\|b'-b\|) + \sum_{\substack{b,b'\in B:\\ \|b'-b\|\ne 0, \|b\|\ne\|b'\|}} \eta(\|b'-b\|) \right).\\
\end{align*}
Since the value in the parenthesis of the second term above is $\sum_{b,b'\in B} 1 = |B|^2$ and the value in the parenthesis of the third term above is $\sum_{b,b'\in B: \|b'-b\|\ne 0 }\eta(\|b'-b\|),$ we see that
\begin{align*} M&= q^{-1} \eta(-j) G_1^{d+1} |A| \left( \sum_{\substack{b,b'\in B:\\ \|b'-b\|=0, \|b\|=\|b'\|}}1 + \sum_{\substack{b,b'\in B:\\ \|b'-b\|\ne 0, \|b\|\ne\|b'\|,\\4j\|b'-b\|=(\|b'\|-\|b\|)^2}} 1\right)\\
&-q^{-2}\eta(-j)G_1^{d+1} |A||B|^2 
-q^{-2}\eta{(-1)}G_1^{d+1} |A|  \sum_{\substack{b,b'\in B,:\\ \|b'-b\|\ne 0}} \eta(\|b'-b\|).\end{align*}

We claim that under the assumptions of Theorem \ref{mainRSO} we have
\begin{equation}\label{claim1} \eta(-j)G^{d+1} = -q^{\frac{d+1}{2}} <0.\end{equation}
Let us assume that this claim holds, which shall be proved in the end of this subsection.
Notice that our claim implies that the first term above for $M$ is negative so that
\begin{align*} M&\le -q^{-2}\eta(-j)G_1^{d+1} |A||B|^2 
-q^{-2}\eta{(-1)}G_1^{d+1} |A|  \sum_{\substack{b,b'\in B,:\\ \|b'-b\|\ne 0}} \eta(\|b'-b\|)\\
 &\le 2q^{\frac{d-3}{2}}|A||B|^2.\end{align*}
This estimate  with \eqref {Thangnonzero} and  \eqref{BForm} yield that
$$ |\Delta(A,B)|\ge \frac{|A|^2|B|^2}{\frac{|A|^2|B|^2}{q} + q^{d-1}|A||B| + 2 q^{\frac{d-3}{2}}|A||B|^2}.$$
This clearly implies that 
$$ |\Delta(A,B)|\ge \frac{1}{4} \min\left\{q, \frac{|A||B|}{q^{d-1}}, 
\frac{|A|}{q^{\frac{d-3}{2}}} \right\}.$$
by which the statement of Theorem \ref{mainRSO} follows. Hence, to complete the proof of Theorem \ref{mainRSO}, it suffices to prove the claim \eqref{claim1}. To do this, we need the following consequence which can be deduced by using   Lemma \ref{ExplicitGauss} and the facts that 
$\eta(-1)=-1$ for $q\equiv 3 \mod 4$ and $\eta(-1)=1$ for $q\equiv 1 \mod 4.$

\begin{lemma}\label{LemGE} Suppose that $q$ is a power of odd prime.
If $d=4k-1$ for some $k\in \mathbb N$ and $q\equiv 3 \mod 4,$ then 
\begin{equation}\label{First} \eta(-1) G_1^{d+1}=-q^{\frac{d+1}{2}}.\end{equation}
On the other hand, if $d=4k+1$ for some $k\in \mathbb N$ or if  $d=4k-1$ for some $k\in \mathbb N$ and $q\equiv 1 \mod 4,$ then
\begin{equation}\label{Second}  \eta(-1) G_1^{d+1}=q^{\frac{d+1}{2}}.\end{equation}
Here we recall that $\eta$ denotes the quadratic character of $\mathbb F_q^*$ and  $G_1$ denotes the standard Gauss sum.
\end{lemma}

\begin{proof} 

The proof of the first part of the lemma is clear, because  $\eta(-1)=-1$ for $q\equiv 3 \mod 4$ and $G_1^{d+1}=G_1^{4k} ={(q^2)}^k=q^{(d+1)/2}$ by Lemma \ref{ExplicitGauss}. \\

 To prove the second part of the lemma, let 
$q=p^\ell$ for an odd prime $p$ and $\ell \in {\mathbb N}.$ It is clear that
$q\equiv 1 \mod 4$ if and only if $p\equiv 1 \mod 4,$ or both $\ell$ is even and $p\equiv 3 \mod 4.$ We also see that $q\equiv 3 \mod 4$ if and only if $p\equiv 3 \mod 4$ and $\ell$ is odd. \\

\textbf{Case 1:} Assume that $d=4k+1$ for some $k\in \mathbb N.$ \\
If $q\equiv 1 \mod 4,$ then $\eta(-1)=1.$ In particular, if $q=p^\ell$ with $p\equiv 1 \mod 4,$ then by Lemma \ref{ExplicitGauss},
$$\eta(-1)G_1^{d+1}=G_1^{4k+2}=((-1)^{\ell-1} q^{1/2})^{4k+2} = q^{2k+1}=q^{(d+1)/2}$$
as desired. On the other hand, if $q=p^\ell$ for even $\ell$ with $p\equiv 3 \mod 4,$ then it follows by Lemma \ref{ExplicitGauss} that 
$$\eta(-1)G_1^{d+1}=G_1^{4k+2}=((-1)^{\ell-1} i^\ell q^{1/2})^{4k+2} = i^{\ell(4k+2)}q^{2k+1}=q^{2k+1}=q^{(d+1)/2}$$
as desired. Thus $\eta{(-1)}G^{d+1}=q^{(d+1)/2}$ for $q\equiv 1 \mod 4.$\\
Now suppose that $q\equiv 3 \mod 4.$ Then $\eta(-1)=-1$, and $q=p^\ell$ for odd $\ell$ with $p\equiv 3 \mod 4.$ Hence by Lemma \ref{ExplicitGauss} we see that
$$ \eta(-1)G_1^{d+1}=(-1) G_1^{4k+2} = (-1) ((-1)^{\ell-1} i^\ell q^{1/2})^{4k+2} = (-1) i^{\ell (4k+2)} q^{2k+1} =q^{2k+1}=q^{(d+1)/2},$$
as desired. This proves the second part of the lemma in the case when $d=4k+1$ for some $k\in \mathbb N.$\\

\textbf{Case 2:} Assume that $d=4k-1$ for some $k\in \mathbb N$ and $q\equiv 1 \mod 4.$
Then $\eta(-1)=1$ and $d+1=4k.$ Thus,
$$ \eta(-1)G_1^{d+1}=G_1^{4k}=(q^2)^k=q^{(d+1)/2},$$
where we used the fact from Lemma \eqref{ExplicitGauss} that $G_1^4=q^2.$ This proves the second part of the lemma in the case when  $d=4k-1$ for some $k\in \mathbb N$ and $q\equiv 1 \mod 4.$
\end{proof}
Now we return to the proof of  the claim \eqref{claim1}. Since $\eta(j)=1$ for any square number $j\in \mathbb F_q^*$ and $\eta(j)=-1$ for any non-square number $j\in \mathbb F_q^*$, the claim \eqref{claim1} follows immediately from Lemma \ref{LemGE}. Thus we complete the proof of Theorem \ref{mainRSO}.

\paragraph{Proof of Theorem \ref{mainRSE}:}
Since $\eta^d\equiv 1$ for even $d$, it follows from Lemma \ref{SForm} that
\begin{equation}\label{SjI} \sum_{t\in \mathbb F_q} \mu^2(t) \le \frac{|A|^2|B|^2}{q} + q^{d-1}|A||B| + I, \end{equation}
where $I$ is defined by
$$ I:=q^{-2}G_1^d |A| \sum_{b,b'\in B, s,r\ne 0} \chi\left(jr+ \frac{s^2\|b'-b\|}{r}\right) \chi(s(\|b'\|-\|b\|).$$
Our main task is to obtain an upper bound of $I.$ We shall proceed as in the proof of Theorem \ref{mainRSO}. As before, we decompose $I$ into four summations:
$$ I:=I_1 + I_2 + I_3+ I_4,$$
where
$$ I_1:=q^{-2}G_1^d |A| \sum_{\substack{b,b'\in B, s,r\ne 0:\\ \|b'-b\|=0, \|b\|=\|b'\|}} \chi\left(jr+ \frac{s^2\|b'-b\|}{r}\right) 
\chi(s(\|b'\|-\|b\|),$$

$$ I_2:=q^{-2}G_1^d |A| \sum_{\substack{b,b'\in B, s,r\ne 0:\\ \|b'-b\|=0, \|b\|\ne\|b'\|}} \chi\left(jr+ \frac{s^2\|b'-b\|}{r}\right) \chi(s(\|b'\|-\|b\|),$$

$$ I_3:=q^{-2}G_1^d |A| \sum_{\substack{b,b'\in B, s,r\ne 0:\\ \|b'-b\|\ne 0, \|b\|=\|b'\|}} \chi\left(jr+ \frac{s^2\|b'-b\|}{r}\right) \chi(s(\|b'\|-\|b\|),$$

$$ I_4:=q^{-2}G_1^d |A| \sum_{\substack{b,b'\in B, s,r\ne 0:\\ \|b'-b\|\ne 0, \|b\|\ne\|b'\|}} \chi\left(jr+ \frac{s^2\|b'-b\|}{r}\right) \chi(s(\|b'\|-\|b\|).$$
Note that the term $I_1$ is the same as
$$ q^{-2}G_1^d |A| \sum_{\substack{b,b'\in B, s,r\ne 0:\\ \|b'-b\|=0, \|b\|=\|b'\|}} \chi(jr).$$
Since $j\ne 0$,  the orthogonality of $\chi$ shows that
$$ I_1=-q^{-2}G_1^d (q-1) |A|\sum_{\substack{b,b'\in B:\\ \|b'-b\|=0, \|b\|=\|b'\|}} 1.$$
Namely, we have
\begin{equation} \label{eqI1}
I_1=-q^{-1}G_1^d |A|\sum_{\substack{b,b'\in B:\\ \|b'-b\|=0, \|b\|=\|b'\|}} 1 + q^{-2}G_1^d |A|\sum_{\substack{b,b'\in B:\\ \|b'-b\|=0, \|b\|=\|b'\|}} 1.
\end{equation}
 Similarly, we see that  
\begin{equation}\label{eqI2}
I_2=q^{-2}G_1^d |A|\sum_{\substack{b,b'\in B:\\ \|b'-b\|=0, \|b\|\ne\|b'\|}} 1.
\end{equation}

The term $I_3$ can be written as
$$ I_3=q^{-2}G_1^d |A| \sum_{\substack{b,b'\in B, s,r\ne 0:\\ \|b'-b\|\ne 0, \|b\|=\|b'\|}} \chi\left(jr+ \frac{s^2\|b'-b\|}{r}\right).$$
Since $\sum_{s\ne 0} = \sum_{s\in \mathbb F_q} - \sum_{s=0},$ we have
$$ I_3= q^{-2}G_1^d |A| \sum_{\substack{b,b'\in B:\\ \|b'-b\|\ne 0, \|b\|=\|b'\|}} \sum_{r\ne 0, s\in \mathbb F_q} \chi\left(jr+ \frac{s^2\|b'-b\|}{r}\right) $$
$$-q^{-2}G_1^d |A| \sum_{\substack{b,b'\in B,:\\ \|b'-b\|\ne 0, \|b\|=\|b'\|}} \sum_{r\ne 0} \chi\left(jr)\right).$$
In order to compute the first term above, we use the basic Gauss sum estimates in \eqref{formb}. To compute the second term above, we notice that  the sum over $r\ne 0$ is $-1,$ because $j\ne 0.$ It follows that
\begin{equation}\label{eqI3} 
I_3= q^{-2}G_1^{d+2} |A| \sum_{\substack{b,b'\in B,:\\ \|b'-b\|\ne 0, \|b\|=\|b'\|}} \eta(j\|b'-b\|) + q^{-2} G_1^d |A| \sum_{\substack{b,b'\in B,:\\ \|b'-b\|\ne 0, \|b\|=\|b'\|}} 1.
\end{equation} 
We move to the estimate of the term $I_4.$ As in the estimate of $I_3$, considering $\sum_{s\ne 0}=\sum_{s\in \mathbb F_q} - \sum_{s=0},$ we can write
$$ I_4= q^{-2}G_1^d |A| \sum_{\substack{b,b'\in B:\\ \|b'-b\|\ne 0, \|b\|\ne\|b'\|}} \sum_{r\ne 0} \chi(jr) \sum_{s\in \mathbb F_q} \chi\left( \frac{s^2\|b'-b\|}{r}+ s(\|b'\|-\|b\|)\right) $$
$$-q^{-2}G_1^d |A| \sum_{\substack{b,b'\in B:\\ \|b'-b\|\ne 0, \|b\|\ne\|b'\|}} \sum_{r\ne 0} \chi(jr).$$
The sum over $s\in \mathbb F_q$ in the first term above can be computed by the Gauss sum formula in \eqref{GRF}. The sum over $r\ne 0$ in the second term above is $-1,$ because $j\ne 0.$ Thus we see that
$$ I_4= q^{-2}G_1^{d+1} |A| \sum_{\substack{b,b'\in B:\\ \|b'-b\|\ne 0, \|b\|\ne\|b'\|}} \sum_{r\ne 0} \chi(jr) \eta\left(\frac{\|b'-b\|}{r}\right)  \chi\left( \frac{ (\|b'\|-\|b\|)^2 r}{-4\|b'-b\|}\right)$$
$$+ q^{-2}G_1^d |A| \sum_{\substack{b,b'\in B:\\ \|b'-b\|\ne 0, \|b\|\ne\|b'\|}} 1$$
$$=q^{-2}G_1^{d+1} |A| \sum_{\substack{b,b'\in B:\\ \|b'-b\|\ne 0, \|b\|\ne\|b'\|}} \sum_{r\ne 0}  \eta\left(\frac{\|b'-b\|}{r}\right)  \chi\left(\left[j-\frac{ (\|b'\|-\|b\|)^2 }{4\|b'-b\|}\right]r\right)$$
$$+ q^{-2}G_1^d |A| \sum_{\substack{b,b'\in B:\\ \|b'-b\|\ne 0, \|b\|\ne\|b'\|}} 1.$$
Notice by the orthogonality of $\eta$ that the sum over $r\ne 0$ of the first term above is zero if $4j\|b'-b\|-(\|b'\|-\|b\|)^2=0.$ When $4j\|b'-b\|-(\|b'\|-\|b\|)^2\ne 0$, we use the second formula in \eqref{formb} to compute the sum over $r\ne 0$ of the first term above. Then we obtain that
\begin{align*} I_4=&q^{-2}G^{d+2}|A| \sum_{\substack{b,b'\in B:\\ \|b'-b\|\ne 0, \|b\|\ne\|b'\|\\4j\|b'-b\|\ne (\|b'\|-\|b\|)^2 }} \eta\left(4j\|b'-b\|
-(\|b'\|-\|b\|)^2\right)\\
&+q^{-2}G_1^d |A| \sum_{\substack{b,b'\in B:\\ \|b'-b\|\ne 0, \|b\|\ne\|b'\|}}1.\end{align*}
We sum this term $I_4$ and terms $I_1, I_2, I_3$ of \eqref{eqI1}, \eqref{eqI2}, \eqref{eqI3}, respectively, and simplify the sums. More precisely, the term $I$ can be written in the following form.

\begin{align*} I=&I_1+I_2+I_3+I_4\\
=& -q^{-1}G_1^d |A|\sum_{\substack{b,b'\in B:\\ \|b'-b\|=0, \|b\|=\|b'\|}} 1 \\
&+ q^{-2}G_1^d |A| \left(\sum_{\substack{b,b'\in B:\\ \|b'-b\|=0, \|b\|=\|b'\|}} 1 
+ \sum_{\substack{b,b'\in B:\\ \|b'-b\|=0, \|b\|\ne\|b'\|}} 1
+ \sum_{\substack{b,b'\in B,:\\ \|b'-b\|\ne 0, \|b\|=\|b'\|}} 1
+ \sum_{\substack{b,b'\in B:\\ \|b'-b\|\ne 0, \|b\|\ne\|b'\|}}1\right)\\
&+ q^{-2}G_1^{d+2} |A| \left( \sum_{\substack{b,b'\in B,:\\ \|b'-b\|\ne 0, \|b\|=\|b'\|}} \eta(j\|b'-b\|) + \sum_{\substack{b,b'\in B:\\ \|b'-b\|\ne 0, \|b\|\ne\|b'\|\\4j\|b'-b\|\ne (\|b'\|-\|b\|)^2 }} \eta\left(4j\|b'-b\|-(\|b'\|-\|b\|)^2\right)  \right).
\end{align*}
Observe that the value in the parenthesis of the second term above is exactly $\sum_{b,b'\in B} 1 = |B|^2$ and  the third term above is dominated by 
$$q^{-2} |G_1|^{d+2} |A| \left( \sum_{\substack{b,b'\in B,:\\ \|b'-b\|\ne 0, \|b\|=\|b'\|}} 1 + \sum_{\substack{b,b'\in B:\\ \|b'-b\|\ne 0, \|b\|\ne\|b'\|\\4j\|b'-b\|\ne (\|b\|-\|b'\|)^2 }} 1\right)$$
$$\le  q^{-2} |G_1|^{d+2} |A| \sum_{b,b'\in B} 1 = q^{\frac{d-2}{2}}|A||B|^2.$$ 
Hence, we have
$$ I\le  -q^{-1}G_1^d |A|\sum_{\substack{b,b'\in B:\\ \|b'-b\|=0, \|b\|=\|b'\|}} 1  + q^{-2}G_1^d |A||B|^2 +  q^{\frac{d-2}{2}}|A||B|^2.$$
Since $G_1^d$ is a real number for even $d$,  the first two terms above are real numbers with different sign so that the sum of them is dominated by
$$ \max\left\{q^{-1} |G_1|^d |A| \sum_{\substack{b,b'\in B:\\ \|b'-b\|=0, \|b\|=\|b'\|}} 1, ~~ q^{-2}|G_1|^d |A||B|^2 \right\} \le q^{\frac{d-2}{2}}|A||B|^2.$$

Thus, we conclude that $I\le 2 q^{\frac{d-2}{2}}|A||B|^2$. It follows from  \eqref{SjI} that
$$ \sum_{t\in \mathbb F_q} \mu^2(t) \le \frac{|A|^2|B|^2}{q} + q^{d-1}|A||B| + 2 q^{\frac{d-2}{2}}|A||B|^2.$$
Now by \eqref{BForm} this implies that
$$ |\Delta(A,B)|\ge \frac{|A|^2|B|^2}{\frac{|A|^2|B|^2}{q} + q^{d-1}|A||B| + 2 q^{\frac{d-2}{2}}|A||B|^2},$$
which in turn implies that
$$ |\Delta(A,B)|\ge \frac{1}{4} \min\left\{q, \frac{|A||B|}{q^{d-1}}, 
\frac{|A|}{q^{\frac{d-2}{2}}} \right\}.$$
This completes the proof of Theorem \ref{mainRSE}. $\hfill\square$
\paragraph{Proof of Theorems \ref{mainRZS}:} Since $\eta^d\equiv 1$ for even $d$ and $j=0,$ we see from Lemma \ref{SForm} that 
\begin{equation}\label{Thangzero} \sum_{t\in \mathbb F_q} \mu^2(t) \le \frac{|A|^2|B|^2}{q} + q^{d-1}|A||B| + R,\end{equation}
where
$R:= q^{-2} G_1^d |A| \sum_{b,b'\in B, s,r\ne 0} \chi\left(\frac{s^2\|b'-b\|}{r}\right) \chi(s(\|b'\|-\|b\|).$

We shall estimate the upper bound of $R.$ First, we make an observation that $G_1^d=-q^{d/2},$ which follows by combining Lemma \ref{ExplicitGauss} with our hypotheses that $d=4k+2$ for $k\in \mathbb N$ and $q\equiv 3 \mod 4.$ For each fixed $s\in \mathbb F_q^*,$ a change of variables, $r'=s^2/r$, gives
$$ R=-q^{-2}q^{d/2} |A| \sum_{b,b'\in B} \left(\sum_{s\ne 0} \chi(s(\|b\|-\|b'\|) \right) \left(\sum_{r'\ne 0} \chi(r'\|b'-b\|)\right).$$
As before, we decompose the sum over $b,b'\in B$ into the following four parts:
\begin{equation}\label{Decom4} \sum_{b,b'\in B} = \sum_{\substack{b,b'\in B:\\ \|b'-b\|=0, \|b\|=\|b'\|}} + \sum_{\substack{b,b'\in B:\\ \|b'-b\|=0, \|b\|\ne\|b'\|}}
+\sum_{\substack{b,b'\in B:\\ \|b'-b\|\ne 0, \|b\|=\|b'\|}}+\sum_{\substack{b,b'\in B: \\
\|b'-b\|\ne 0, \|b\|\ne \|b'\|}}.\end{equation}
Let $R=R_1+R_2+R_3+R_4$ where each $R_j$ denotes the contribution for the value R, which is restricted to the $j$-th summation in \eqref{Decom4}. Let us estimate each $R_j.$
By the orthogonality of $\chi,$ 
$$ R_1=-q^{\frac{d-4}{2}} (q-1)^2 |A|\sum_{\substack{b,b'\in B:\\ \|b'-b\|=0, \|b\|=\|b'\|}} 1,$$
which is negative and so we have
$$ R\le R_2+R_3+R_4.$$
Similarly, the orthogonality of $\chi$ shows that
$$ R_4=-q^{\frac{d-4}{2}} |A|\sum_{\substack{b,b'\in B: \\
\|b'-b\|\ne 0, \|b\|\ne \|b'\|}} 1,$$
which is also negative and so we see that
$$ R\le R_2+R_3.$$
By the orthogonality of $\chi$ it is not hard to see that
$$ R_2=q^{\frac{d-4}{2}} (q-1)|A| \sum_{\substack{b,b'\in B:\\ \|b'-b\|=0, \|b\|\ne\|b'\|}}1 \le q^{\frac{d-2}{2}}|A| \sum_{\substack{b,b'\in B:\\ \|b'-b\|=0, \|b\|\ne\|b'\|}}1,$$
and
$$R_3\le q^{\frac{d-2}{2}}|A|\sum_{\substack{b,b'\in B:\\ \|b'-b\|\ne 0, \|b\|=\|b'\|}} 1.$$
This implies that
$$ R_2+R_3 \le q^{\frac{d-2}{2}}|A| \left(  \sum_{\substack{b,b'\in B:\\ \|b'-b\|=0, \|b\|\ne\|b'\|}}1 +\sum_{\substack{b,b'\in B:\\ \|b'-b\|\ne 0, \|b\|=\|b'\|}} 1\right) \le q^{\frac{d-2}{2}}|A| \sum_{b,b'\in B} 1 =q^{\frac{d-2}{2}}|A||B|^2. $$

Hence, we conclude that
$$ R\le q^{\frac{d-2}{2}}|A||B|^2.$$
Combining this estimate with \eqref{Thangzero}, it follows by the formula \eqref{BForm} that
$$ |\Delta(A,B)|\ge \frac{|A|^2|B|^2} {\frac{|A|^2|B|^2}{q} + q^{d-1}|A||B|+q^{\frac{d-2}{2}}|A||B|^2}.$$
Therefore, we have
$$ |\Delta(A,B)|\ge \frac{1}{3}\min\left\{q,~\frac{|A||B|}{q^{d-1}},~ \frac{|A|}{q^{\frac{d-2}{2}}} \right\},$$ from which the statement of Theorem \ref{mainRZS} follows.

\section{Constructions}\label{sec15}
Let $V$ be a set of vectors in $\mathbb{F}_q^d$. Suppose that $V=\{v_1, \ldots, v_k\}$. We say that vectors in $V$ are mutually orthogonal if $v_i\cdot v_j=0$ for any $1\le i, j\le k$. The following lemma is taken from \cite{hart}. 

\begin{lemma}\label{oth-tho}
Let $\mathbb{F}_q$ be a finite field of order $q$. If $d=4k$ with $k\in \mathbb{N}$, then there always exist $d/2$ mutually orthogonal vectors. If $d=4k+2$ and $q\equiv 1\mod 4$, then there also exist $d/2$ mutually orthogonal vectors.
\end{lemma}

We note that if $d=4k+2$ and $q\equiv 3\mod 4$, then it is impossible to have $d/2$ mutually orthogonal vectors in $\mathbb{F}_q^d$.
\subsection{Sharpness of Theorem \ref{mainRP}}

In the following lemma, we show that the condition $|A||B|\ge q^d$ in Theorem \ref{mainRP} is optimal. 

\begin{lemma}\label{para-3} Let $P$ be the paraboloid in $\mathbb{F}_q^d$.  For any $\epsilon>0$, the following consequences hold:\\
\begin{enumerate}
\item For $d=4k-1$, $k\in \mathbb N$, and $q\equiv 3 ~\mbox{mod}~4$, there exist $A\subset P$ and $B\subset \mathbb{F}_q^d$ such that $|A||B|\sim q^{d-\epsilon}$ and
$$ |\Delta(A,B)|\sim q^{1-\epsilon}.$$
\item For $d\ge 4$ even and $q\equiv 1\mod 4$, there exist $A\subset P$ and $B\subset \mathbb{F}_q^d$ such that $|A||B|\sim q^{d-\epsilon}$ and
$$ |\Delta(A,B)|\sim q^{1-\epsilon}.$$
\item For $d=4k+2$ and $q\equiv 3\mod 4$, there exist $A\subset P$ and $B\subset \mathbb{F}_q^d$ such that $|A||B|\sim q^{d-\epsilon}$ and
$$ |\Delta(A,B)|\sim q^{1-\epsilon}.$$

\end{enumerate}
\end{lemma}
\begin{proof}
\begin{enumerate}
\item Since $d=4k-1$ and $q\equiv 3\mod 4$, by Lemma \ref{oth-tho}, there exists $(d-3)/2$ vectors, denoted by $v_1, \ldots, v_{(d-3)/2}$, that are mutually orthogonal in the space $\mathbb{F}_q^{d-3}\times \{0\}\times \{0\}\times \{0\} \subset \mathbb{F}_q^d$. Define 
\[A=Span(v_1, \ldots, v_{(d-3)/2}).\]
It is clear that $A$ is a set on $P$. Let $R$ be any set in $\mathbb{F}_q$ of size $q^{1-\epsilon}$. For $r\in R$,  let $S(0, 0, 0, r)$ be the sphere of radius $r$ centered at $(0, 0, 0)\in \mathbb{F}_q^3$. Define 
\[B=Span(v_1, \ldots, v_{(d-3)/2})+\{(0, \ldots, 0, x, y, z)\colon (x, y, z)\in S(0, 0,0, r), r\in R\}.\]
By a direct computation, we have $|B|\sim q^{\frac{d-3}{2}+3-\epsilon}$. Thus, we obtain $|A||B|\sim q^{d-\epsilon}$. On the other hand, it is clear that $\Delta(A, B)=R$. This gives us 
\[|\Delta(A, B)|=q^{1-\epsilon}.\]
\item The proofs of the second and third statements are almost identical except we use $(d-2)/2$ mutually orthogonal vectors. 
\end{enumerate}

\end{proof}
\subsection{Sharpness of Theorem \ref{mainRSO}}

Our next lemma tells us that the condition $|A||B|>q^d$ in Theorem \ref{mainRSO} is optimal.
 
\begin{lemma}\label{sphere-c-1}
Let $S_j$ be the sphere of radius $j\ne 0$ centered at the origin in $\mathbb{F}_q^d$. For any $\epsilon>0$, the following two statements hold:
\begin{enumerate}
\item  Let $j$ be  a square number of $\mathbb F_q^*.$ For $d=4k-1$, $k\in \mathbb N,$ and  $q\equiv 3 \mod 4$, there exist $A\subset S_j$ and $B\subset \mathbb{F}_q^d$ such that $|A||B|=q^{d-\epsilon}$ and $|\Delta(A, B)|\sim q^{1-\epsilon}.$
\item   Let $j$ be not a square number of $\mathbb F_q^*.$ For either $d=4k+1$, $k\in \mathbb N,$ or $d=4k-1$ and $q\equiv 1 \mod 4$, there exist $A\subset S_j$ and $B\subset \mathbb{F}_q^d$ such that $|A||B|=q^{d-\epsilon}$ and $|\Delta(A, B)|\sim q^{1-\epsilon}$.
\end{enumerate}
\end{lemma}
\begin{proof}
The proofs of two statements are similar,  we only prove to give the proof of the first one.  Since $d=4k-1$ and $q\equiv 3\mod 4$, by Lemma \ref{oth-tho}, there exist $(d-3)/2$ vectors, denoted by $v_1, \ldots, v_{(d-3)/2}$ which are mutually orthogonal in the space $\mathbb{F}_q^{d-3}\times 0\times 0\times 0 \subset \mathbb{F}_q^d$. Define 
\[A=Span(v_1, \ldots, v_{(d-3)/2})+(0, \ldots, 0, 0, 0, j^{1/2}).\]
It is clear that $A$ is a set on $S_j$ since $j$ is a square number. 

Let $R$ be any set of $\mathbb{F}_q$ of size $q^{1-\epsilon}$. For $r\in R$,  let $S(0, 0, j^{1/2}, r)$ be the sphere of radius $r$ centered at $(0, 0, j^{1/2})\in \mathbb{F}_q^3$. Define 
\[B=Span(v_1, \ldots, v_{(d-3)/2})+\{(0, \ldots, 0, x, y, z)\colon (x, y, z)\in S(0, 0, j^{1/2}, r), r\in R\}.\]
By a direct computation, we have $|B|\sim q^{\frac{d-3}{2}+3-\epsilon}$. Thus, we obtain $|A||B|\sim q^{d-\epsilon}$. On the other hand, it is clear that $\Delta(A, B)=R$. This gives us 
\[|\Delta(A, B)|=q^{1-\epsilon}.\]
\end{proof}
\subsection{Sharpness of Theorem \ref{mainRSE}}
\begin{lemma}\label{sphere-c-3}
Let $S_j$ be the sphere of radius $j\ne 0$ centered at the origin in $\mathbb{F}_q^d$ with $d\ge 4$ even. For any $\epsilon>0$, there exist $A\subset S_j$ and $B\subset \mathbb{F}_q^d$ such that $|A||B|=q^{d-\epsilon}$ and $|\Delta(A, B)|\sim q^{1-\epsilon}$. 
\end{lemma}
\begin{proof}
The proof of Lemma \ref{sphere-c-3} is similar to that of Lemma \ref{sphere-c-1} except that instead of using $\frac{d-3}{2}$ mutually orthogonal vectors, we use $\frac{d-4}{2}$ vectors for the case $d=4k$ and $q\equiv 3\mod 4$, and we use $\frac{d-2}{2}$ vectors in the following two cases: $(d=4k+2)$ and $(d=4k, ~q\equiv 1\mod 4)$.
\end{proof}
\subsection{Sharpness of Theorem \ref{mainRZS}}
\begin{lemma}\label{sphere-c-3-2} Let $S_0$ be the sphere of zero radius centered at the origin in $\mathbb{F}_q^d$. Suppose that $d=4k+2$, $k\in \mathbb{N}$, and $q\equiv 3\mod 4$. For any $\epsilon>0$, there exist $A\subset S_0$ and $B\subset \mathbb{F}_q^d$ such that $|A||B|=q^{d-\epsilon}$ and $|\Delta(A, B)|\sim q^{1-\epsilon}$. 
\end{lemma}
\begin{proof}
Define 
\[A=Span(v_1, \ldots, v_{(d-2)/2})\subset \mathbb{F}_q^{d-2}\times \{0\}\times \{0\}.\]
It is clear that $A$ is a set on $S_0$ and $|A|=q^{(d-2)/2}$.  Let $R$ be any set of $\mathbb{F}_q$ of size $q^{1-\epsilon}$. For $r\in R$,  let $S(0, 0, r)$ be the sphere of radius $r$ centered at $(0, 0)$ in $\mathbb{F}_q^2$. Define 
\[B=Span(v_1, \ldots, v_{(d-2)/2})+\{(0, \ldots, 0, x, y)\colon (x, y)\in S(0, 0, r), r\in R\}.\]
By a direct computation, we have $|B|\sim q^{\frac{d-2}{2}+2-\epsilon}$. Thus, we obtain $|A||B|\sim q^{d-\epsilon}$. On the other hand, it is clear that $\Delta(A, B)=R$. This gives us 
\[|\Delta(A, B)|=q^{1-\epsilon}.\]
\end{proof}
%

\section*{Acknowledgement}
Thang Pham was supported by Swiss National Science Foundation grant P2ELP2175050. 
Doowon Koh was supported by the National Research Foundation of Korea grant 
NRF-2018R1D1A1B07044469.

\end{document}